\documentclass[a4paper,12pt]{amsart}
\usepackage{cancel}

\usepackage{xcolor}
\usepackage{hyperref}
\usepackage{bm}
\usepackage{mathtools}
\usepackage{graphicx}
\usepackage{amsmath,amssymb,amsthm,amsfonts}
\usepackage{amssymb}
\usepackage[active]{srcltx}

\usepackage{geometry}
\newtheorem{thm}{Theorem}[section]

\newtheorem{cor}[thm]{Corollary}
\newtheorem{prop}[thm]{Proposition}
\newtheorem{lem}[thm]{Lemma}

\theoremstyle{definition}
\newtheorem{defn}[thm]{Definition}

\theoremstyle{remark}
\newtheorem{rem}[thm]{Remark}

\usepackage{tikz}
\usepackage{standalone}
\usepackage{graphicx}
\usepackage{subcaption}
\usepackage{multicol,float}

\usepackage{mathtools}
\newcommand{\disp}{\displaystyle} 
\usepackage{bm}
\newcommand{\R}{\mathbb{R}}

\newcommand{\rnu}{\mathbb{R}^{N-1}}

\newcommand{\dr}{\partial}

\numberwithin{equation}{section}

\title{Serrin's overdetermined problems on epigraphs}

\author{Nicolas Beuvin$^{\dagger}$, Alberto Farina$^{\dagger}$}
\address{$^{\dagger}$  Universit\'e de Picardie Jules Verne, LAMFA, CNRS UMR 7352, 33, rue Saint-Leu 80039 Amiens, France}
\email{nicolas.beuvin@u-picardie.fr, alberto.farina@u-picardie.fr}

\begin{document}




\begin{abstract}
In this work we establish some rigidity results for Serrin's overdetermined problem 
$$
\left\{
	\begin{array}{cll}
		- \Delta u=f(u) & \text{in}& \Omega,\\
		u > 0& \text{in} & \Omega,\\
		u=0 & \text{on} & \dr\Omega,\\
		\dfrac{\partial u}{\partial \eta} = \mathfrak{c} = const. & \text{on} & \dr \Omega,
	\end{array}
	\right.
$$
when $\Omega \subset \R^N$ is an epigraph (not necessarily globally Lipschitz-continuous) and $u$ is a classical solution, possibly unbounded. \\
In broad terms, our main results prove that $\Omega$ must be an affine half-space and $u$ must be one-dimensional, provided the epigraph is bounded from below. These results hold when $f$ is of Allen-Cahn type and $ N \geq 2$ or, alternatively, when $f$ is locally Lipschitz-continuous (with no restriction on the sign of $f(0)$) and  $ N \leq 3$. 
Finally, when $f(0) <0$, we also prove a new monotonicity result, valid in any dimension $ N \geq 2$. 
\end{abstract}
\maketitle
\section{Introduction}

In his seminal paper \cite{ser}, published in 1971, J. Serrin proved that if $\Omega$ is a bounded and smooth domain for which there is a smooth solution to the overdetermined problem 

\begin{equation}\label{probleme}
	\left\{
	\begin{array}{cll}
		- \Delta u=f(u) & \text{in}& \Omega,\\
		u > 0& \text{in} & \Omega,\\
		u=0 & \text{on} & \dr\Omega,\\
		\dfrac{\partial u}{\partial \eta} = \mathfrak{c} = const. & \text{on} & \dr \Omega,
	\end{array}
	\right.
\end{equation} 

\noindent then $\Omega$ must be a ball and $u$ is radially symmetric about its center. \\
Here, $\eta$ denotes the outward unit normal at $\partial \Omega$ and $f$ is a function of class $C^{1}$.

\smallskip

In this paper we are concerned with Serrin's overdetermined problem \eqref{probleme} on epigraphs, i.e., on unbounded domains of the form 
\begin{equation}
	\Omega:=\lbrace x=(x', x_N) \in \R^N, x_N> g(x') \rbrace, \label{epigraph}
\end{equation}
where $g: \R^{N-1} \to \R$ is a (sufficiently) smooth function and $ N \geq 2$. 

\smallskip

Serrin’s overdetermined problems on epigraphs naturally arise when studying regularity results for certain free-boundary problems using a blow-up procedure. This observation led H. Berestycki, L. Caffarelli, and L. Nirenberg to study, for the first time, Serrin’s overdetermined problems on epigraphs in their article \cite{BCN3}. For functions $f$ of \textit{Allen-Cahn type}\footnote{ \, i.e., $f \in Lip_{loc}([0, + \infty))$ such that :
\begin{equation}\label{Allen-Cahn-type}
\left\{\begin{array}{rl}
&  \quad \exists \,  \mu > \delta_2 > \delta_1>0 \, :\\[0.3cm]
{\rm (I)} & \quad \disp f >0 \quad \text{on } (0,\mu), \qquad f\leq 0  \quad \text{on } [\mu, +\infty),\\[0.3cm]
{\rm (II)} & \quad \disp f(t) \ge \delta_1 t \quad \text{for } t \in (0, \delta_1), \qquad f \text{ is non-increasing on } (\delta_2, \mu).
\end{array}\right.
\end{equation}
The typical representative is given by the classical Allen-Cahn nonlinearity $ f(t) = t-t^3$.\\}, the authors of \cite{BCN3} employed the sliding method to prove the following rigidity result: 

\textit{if $g \in C^2$ is a globally 
Lipschitz-continuous function satisfying an additional condition at infinity,\footnote{\, The additional condition on $g$ required in \cite{BCN3} is :
$ \, \forall \, \tau \in \R^{N-1}, \quad \lim_{{\atop {}_{\atop {\hskip-.6cm {\vert x'\vert \to+\infty}}}}}(g (x' +\tau ) - g(x')) = 0.$}
and problem \eqref{probleme} admits a smooth and bounded solution, then $g$ must be constant $($i.e., $\Omega = \lbrace x \in \R^N \, : \, x_N > const. \rbrace$ is an upper half-space$)$ and $u$ takes the form $u=u(x_N)$.}

In 2010, A. Farina and E. Valdinoci \cite{FVarma} introduced a new geometric tool that allowed them to improve the above result  without any additional technical condition at infinity on $g$, but under the dimensional constraint $N \leq 3$. 
Specifically, they proved the following result : 

\textit{if $f$ is of Allen-Cahn type, $g \in C^3$ is a globally Lipschitz-continuous function and problem \eqref{probleme} admits a bounded $C^2$-solution, then $\Omega $ is an affine half-space and $u$ is one-dimensional.}

\noindent Moreover, in \cite{FVarma} it is proved that :

\textit{i) the same rigidity result is true  if $\Omega$ is a two-dimensional epigraph of class $C^3$, $f$ is any locally Lipschitz-continuous function and $u$ is a monotone solution $($i.e. $\frac{\partial u}{\partial x_{2}}>0$ on $\Omega)$, possibly unbounded, but with bounded gradient;}

\textit{ii) for any locally Lipschitz-continuous function $f$, Serrin's overdetermined problem \eqref{probleme} has no bounded solution on globally Lipschitz-continuous coercive epigraphs if either $N=2$ or $N = 3$ and $f \geq 0$. }

\medskip

Subsequently, the new geometric tool introduced in \cite{FVarma} has been used to study Serrin's overdetermined problem on general smooth domains of a Riemannian manifold (see \cite{FMV}).

\medskip

In \cite{ww}, under more specific conditions on the Allen-Cahn type nonlinear function $f$ (but still including the classical Allen-Cahn model), K. Wang and J. Wei considered Serrin's overdetermined problem \eqref{probleme} on an epigraph $\Omega$ and proved that $\Omega$ must be an affine half-space and $u$ is one-dimensional, in any of the following cases :

\textit{iii) on any smooth epigraph, if $N=2$;}

\textit{iv) on any smooth and globally Lipschitz-continuous epigraph, if $N\geq 3$;}

\textit{v) on any smooth epigraph, if $N \leq 8$ and $u$ is a monotone solution, i.e. $\frac{\partial u}{\partial x_N}>0$ on $\Omega$.}

\noindent In \cite{ww} it is also proved that 

\textit{vi) Serrin's overdetermined problem \eqref{probleme} has no solution on smooth coercive epigraphs if $ N \geq 2$}. 

\smallskip

\noindent The result in item \textit{v)} above is optimal since M. Del Pino, F. Pacard and J. Wei have built a smooth epigraph $\Omega$ (which is not an affine half-space) such that, if $N\geq 9$, the overdetermined problem \eqref{probleme} admits a smooth, bounded and monotone solution (see Theorem 1 in \cite{ppw}). Also note that, the epigraph built in \cite{ppw} is not globally Lipschitz-continuous. 

\noindent The proofs in \cite{ww} involve tools from Geometric Measure Theory and the theory of free boundary problems. 

\medskip

Finally we recall that A. Ros, D. Ruiz and P. Sicbaldi, in their work \cite{RRS17}, proved the following result :

\textit{vii) if $f$ is (locally) Lipschitz-continuous, $\Omega \subset \R^2$ is a $C^{1,\alpha}$ domain whose boundary is unbounded and connected and problem \eqref{probleme} admits a bounded $C^2$-solution with $  \mathfrak{c} <0$, then $\Omega $ is an affine half-plane and $u$ is one-dimensional.} 

In particular, the result above applies when $\Omega $ is a two-dimensional epigraph defined by a function 
$ g \in C^{1,\alpha}_{loc}(\R)$, for some $ \alpha \in (0,1).$

\bigskip 

In view of the above motivation and the previous results, the following problem naturally emerges : 

\smallskip 

\textbf{Problem (P)} \textit{Provide general, natural conditions on the locally Lipschitz-continuous function $f$ and on the $C^1$ function $g$ such that, if the overdetermined problem \eqref{probleme} admits a classical solution,} \text{possibly unbounded},  \textit{then $\Omega$ must be an affine half-space and $u$ is one-dimensional.}

\smallskip 

The necessity to consider also unbounded solutions to the Serrin's overdetermined problem \eqref{probleme} comes from the fact that, for some natural functions $f$, the only solutions to \eqref{probleme} are unbounded. This is well illustrated by the case of harmonic functions, that is when $f \equiv 0$. In this case, there might be no bounded solution to \eqref{probleme}, whereas the function $u(x) =x_N$ is a positive unbounded harmonic function on the half-space $\R^N_+$ with zero Dirichlet boundary condition and Neumann boundary condition equals to $-1$ (see also Proposition  \ref {prop_princ-maxBCN} and Remark \ref{rem-prop_princ-maxBCN}). 

\smallskip 

Before proceeding further, we explicitely note that the results of \cite{BCN3},\cite{FVarma},\cite{RRS17} and \cite{ww} mentioned above, provide a partial answer to the aforementioned problem.

\smallskip 

The results of this article also fall within the research framework outlined by Problem (P). 
They extend and complement the known results mentioned above. In broad terms, our main results prove that $\Omega$ must be an affine half-space and $u$ must be one-dimensional, provided the epigraph is bounded from below. Our results also cover the case of unbounded solutions. When $f(0)\geq0$, they are based on the \textit{monotonicity} results of the recent paper \cite{bfs}, where the authors and B. Sciunzi prove that, when $ \Omega$ is an epigraph bounded from below, any  solution to the Dirichlet problem   
\begin{equation}\label{NonLin-PoissonEq}
\begin{cases}
-\Delta u=f(u) & \text{ in } \Omega,\\
\quad u >0 & \text{ in } \Omega,\\
\quad u=0\,\, &\text{ on } \partial\Omega,
\end{cases}
\end{equation}
is \textit{monotone}, i.e., $u$ satisfies $\frac{\partial u}{\partial x_N}>0$ on $\Omega$. Then we combine the monotonicity of $u$ together with the geometric approach developed by the second author and E.Valdinoci in \cite{FVarma} to prove that the epigraph $\Omega$ is an affine half-space and $u$ is one-dimensional.



The situation when $f(0) <0$ is more involved, and indeed the only known results in this case are restricted to dimension 2 (see results \textit{i)} and \textit{vii)} previously discussed). Furthermore, in this case, we cannot use the monotonicity results recently obtained in \cite{bfs}, since those results relied heavily on the assumption $f(0)\geq0$, which ensured the validity of the strong maximum principle for non-negative solutions to a semilinear Poisson equation.   
In the present work, to address the case $f(0)<0$, we  first  prove  a  new monotonicity result for solutions to the overdetermined problem \eqref{probleme} on epigraphs bounded from below and with non-zero Neumann boundary condition.
Unlike the strategy employed in the recent article \cite{bfs}, it is this latter condition that plays a key role in the proof of this new monotonicity result. From this, and the geometric approach developed in \cite{FVarma}, we deduce that the epigraph $\Omega$ is an affine half-space and $u$ is one-dimensional under the dimension constraint $N \leq 3.$ 

Specifically, in Section \ref{Sect-sACt} we consider Serrin's overdetermined problem \eqref{probleme} for a family of specified Allen-Cahn-type nonlinearities in any dimension $ N \geq2$. In this framework, if $\Omega$ belongs to a large class of epigraphs bounded from below (larger than that of globally Lipschitz-continuous epigraphs) and the overdetermined problem \eqref{probleme} admits a solution, possibly unbounded, then we prove that $\Omega $ is a half-space and $u$ is one-dimensional. 
This is the content of Theorem \ref{th-allen-cahn 2}.  In passing, we also establish a general rigidity result for monotone solutions to \eqref{probleme} (see Theorem \ref{th-allen-cahn 1}). This result interpolates between the above mentioned results \textit{iii)-vi)}.  It is of independent interest and ancillary to Theorem \ref{th-allen-cahn 2}.

In Section \ref{Sect-f(0)non-neg} we consider a general locally Lipschitz-continuous function $f$ satisfying $f(0)\geq 0$ and solutions $u$, bounded or not, to the overdetermined problem \eqref{probleme}. In this framework we prove the rigidity of $\Omega$ (i.e., $\Omega$ is a half-space) and the one-dimensional symmetry of $u$ in dimension $N \leq 3$. See Theorems \ref{Teo2-lip}, \ref{Teo3-lip} and \ref{Teo1-C1}.


Section \ref{Sect-f(0) neg} deals with the case of locally Lipschitz-continuous functions $f$ with $f(0) <0$. 
As previously discussed, we first prove a new monotonicity result valid for any dimension $N \geq 2$ (see Theorem \ref{th_overdet-f(0)<0}) and then we use it to demonstrate the rigidity of $\Omega$ and the one-dimensional symmetry of $u$, 
when $N \leq 3$.  


\smallskip

Several results used to prove our main conclusions are established in Section \ref{Appendix1}.
Some of these general results are of independent interest and could be useful for further studies of the nonlinear Poisson equation on unbounded domains. Also, for the reader's convenience, the notations used in the paper are collected at the end of Section \ref{Appendix1}. 

\smallskip

We end this section by observing that, in the present work, we always consider domains $\Omega$ of class $C^1$ and solutions $u$ belonging to\footnote{We recall (see Subsection \ref{Notations}) that $C^1(\overline{U})$ is the set of functions in $C^1(U)$ all of whose derivatives of order $ \leq 1$ have continuous (not necessarily bounded) extensions to the closure of the open set $U$. In particular we are not assuming that those functions are bounded on $U$.} $C^1(\overline{\Omega}) \cap C^2(\Omega)$. These are the minimal assumptions on $\Omega$ and on $u$ to deal with the problem in a classical framework. We also note that, when $\mathfrak{c} \neq 0$ in \eqref{probleme}, these assumptions automatically ensure that $\Omega$ is of class $C^{2,\alpha}$ and $u \in C^{2,\alpha}_{loc}(\overline{\Omega})$, for any $ \alpha \in (0,1)$ (see Theorem 1 of \cite{Vogel} and observe that the proof of this result is local). These results will therefore be used in our work (perhaps without explicitly recalling them). 

\bigskip

The paper is organized as follows :

\smallskip

\begin{enumerate}

\small {

\item[1.]  \textit{Introduction} 

\smallskip

\item[2.]  \textit{Serrin's overdetermined problem with Allen-Cahn type nonlinearity} 

\smallskip

\item[3.] \textit{Serrin's overdetermined problem with general nonlinearity. The case $f(0) \geq 0$}

\smallskip

\item[4.] \textit{Serrin's overdetermined problem with general nonlinearity. The case $f(0) <0$}

\begin{enumerate}
              \item[•] \textit{A new monotonicity result in any dimension $N \geq 2$}
              
              \item[•] \textit{Rigidity and one-dimensional symmetry results 
              }
              
              \end{enumerate}

\smallskip

\item[5.] \textit{Some auxiliary results}

              \begin{enumerate}
              \item[•] \textit{Some uniform bounds}
              
              \item[•] \textit{Some uniform estimates}
              
              \item[•] \textit{Some energy estimates}
              
              \item[•] \textit{Appendix}
              
              \item[•] \textit{Notations}
              \end{enumerate}
 
}              
\end{enumerate}
 
\section{Serrin's overdetermined problem with Allen-Cahn type nonlinearity}\label{Sect-sACt}

\smallskip

In this section we consider Serrin's overdetermined problem \eqref{probleme} when $f$ is modeled on the classical Allen-Cahn nonlinearity. 
More precisely, we deal with a family of nonlinearities $f$, already considered in \cite{ww}, which we will call \textit{specified Allen-Cahn-type nonlinearities}. 

A function $f$ is a \textit{specified Allen-Cahn-type nonlinearity} if the function 
\begin{equation}\label{double-potentiel} 
	W(t) := \displaystyle\int_{0}^{1} f(s)ds -\displaystyle\int_{0}^{t} f(s)ds, \qquad t \geq 0, 
\end{equation}
belongs to $C^{2}([0,+\infty))$ and satisfies the following conditions :
\begin{itemize}

	\item[$\mathtt{W}1$)] $W \geq 0$, $W(1)=0 \,$ and $\, W>0 \, \text{ in } \, [0,1)$;

     \item[$\mathtt{W}2$)] $W'<0 $ in $(0,1)$, and either $W'(0)\neq 0$ or
	\begin{equation*}
		W'(0)=0 \quad \text{and} \quad W''(0)\neq 0;
	\end{equation*} 

	\item[$\mathtt{W}3$)] $\exists \, \delta_2 \in (0,1),  \exists \, \kappa>0$ such that $W'' \geq \kappa $ in $ [\delta_2,1]$;
	
	\item[$\mathtt{W}4$)] $\exists \, p>1, \exists \, c>0 $ such that $W'(t) \geq c(t-1)^{p}$ for $t>1$.
	
\end{itemize}

\noindent The representative example of this family is given by $f(t) = t-t^3$, the well-known \textit{Allen-Cahn nonlinearity}.\footnote{\, in this case we have $W(t) = \frac{(1-t^2)^2}{4}.$}  

For this family of nonlinear functions, the authors of \cite{ww} have proved results \textit{iii), iv), v) and vi)} previously discussed in the Introduction. Our first theorem interpolates between those results. It applies to monotone solutions, it is of independent interest and, when $N \leq 8$, it recovers Theorem 1.4 in \cite{ww}, that is, result \textit{v)} discussed in the Introduction (see also Remark \ref{rem-th-allen-cahn 1} below).


\begin{thm}\label{th-allen-cahn 1}
Assume $N \geq 2$ and let $f \in C^1([0, + \infty))$ be a specified Allen-Cahn type nonlinearity.
Let $u \in C^1(\overline{\Omega}) \cap C^2(\Omega)$ be a solution to \eqref{probleme} satisfying 
\begin{equation}\label{hyp-monot}
		\frac{\partial u}{\partial x_N}(x)>0 \qquad \forall x \in \Omega.
\end{equation}
Let $\Omega \subset \R^N$ be an epigraph defined by a function $g \in C^1(\R^{N-1})$ and, if $ N \geq 9$ let us also assume  that $g$ satisfies :
\begin{itemize}
	\item[a)] $\exists \, C>0$ for which 
	\begin{equation*}
		g(x') \geq - C(1 + \vert x' \vert) \qquad \forall \, x' \in \R^{N-1};
	\end{equation*} 
\end{itemize}	
or

\smallskip

\begin{itemize}
     \item[b)] $\exists \, C>0$ for which 
	\begin{equation*}
		g(x') \leq  C(1 + \vert x' \vert) \qquad \forall \, x' \in \R^{N-1}.
	\end{equation*} 
\end{itemize}

\medskip

\noindent Then, $\Omega=\R^{N}_{+}$ up to isometry and there exists $u_{0}:[0,+\infty) \to [0,+\infty)$ strictly increasing such that
	\begin{equation}
		u(x)=u_{0}(x_{N}) \qquad \forall x \in \R^{N}_{+}. 
	\end{equation}
\end{thm}

\medskip

Some remarks are in order

\begin{rem}\label{rem-th-allen-cahn 1}

(i) In the above Theorem we do not assume that $u$ is bounded. Indeed, in Corollary \ref{cor_K-O-B-type} we prove that this property holds for any solution to the Dirichlet problem \eqref{NonLin-PoissonEq} on general domains, possibly unbounded, when $f$ is a specified Allen-Cahn type nonlinearity. 

\smallskip


(ii) We also notice that, when $g$ is globally Lipschitz-continuous, the monotonicity assumption \eqref{hyp-monot} is automatically satisfied in any dimension $N \geq 2$ thanks to item (i) above and to the paper \cite{BCN3}. Since the Lipschitz character of $g$ implies the validity of the additional condition a) (actually also b)), we see that the above Theorem also recovers Theorem 1.2 in \cite{ww} (that is, result \textit{iv)} discussed in the Introduction). 

When $g$ is coercive, the monotonicity assumption \eqref{hyp-monot} is automatically satisfied in any dimension $N \geq 2$ thanks to the work \cite{Esteban-Lions} and, since the additional condition a) is clearly in force, we see that also Theorem 1.3 in \cite{ww} (that is, result \textit{vi)} discussed in the Introduction) is a particular case of Theorem \ref{rem-th-allen-cahn 1} above. 

\smallskip

(iii) The additional assumptions for $ N \geq 9$ have the following geometrical interpretation:  either  $\R^N \setminus \Omega$ contains a cone or $\Omega$ contains a cone (or equivalenty, either the graph of $g$ lies above the graph of a cone or  the graph of $g$ lies below the graph of a cone). 

\smallskip

(iv) The above remarks also show that the epigraph constructed by M. Del Pino, F. Pacard and J. Wei in \cite{ppw} does not contain a cone (and neither does its complementary) and therefore that this epigraph is neither globally 
Lipschitz-continuous (as already observed in \cite{ppw}), nor contained in an affine half-space.  
\end{rem}

\smallskip

In view of the discussions in Remark \ref{rem-th-allen-cahn 1} above, the following two research lines seem legitimate and interesting, when $f$ is a \textit{specified Allen-Cahn-type nonlinearities} :

\medskip

1) for $ N \leq 8$, provide natural assumptions on $g$, weaker than the global Lipschitz continuity, which guarantee the monotonicity property;

\smallskip

2) for $ N \geq 9$, give natural assumptions on the epigraph $ \Omega$ ensuring that, if the overdetermined problem \eqref{probleme} admits a solution, then $\Omega $ is necessarily an affine half-space and $u$ is one-dimensional. 


\bigskip

The next result (see Theorem \ref{th-allen-cahn 2} below) takes a step in the directions indicated in items 1) and 2) above.  In general terms, this result proves the rigidity of the epigraph $\Omega$ and the one-dimensional symmetry of the solution $u$ when the epigraph is contained in a half-space and this holds true without any monotonicity assumption on $u$. 
To our knowledge, this is the first result of this type in any dimension $N \geq 3$. 
It stands on the results of the recent paper \cite{bfs}, where the authors and B. Sciunzi prove the monotonicity property \eqref{hyp-monot} for solutions to the Dirichlet problem \eqref{NonLin-PoissonEq} on epigraphs bounded from below, and defined by a function $g$ belonging to a large class $\mathcal{G}$ of continuous functions. Specifically, in  \cite{bfs} we introduced the following 
\begin{defn} \label{Def-classG} Assume $N \geq 2.$ We say that a continuous function $g : \R^{N-1} \mapsto \R$ belongs to the class $\mathcal{G}$, if it satisfies the next  compactness property : 

\medskip
\noindent 
\textit{Any sequence $(g_k)$ of translations of $g$, which is bounded at some fixed point of $\R^{N-1}$, admits a subsequence converging uniformly on every compact sets of $\R^{N-1}$ } 
\end{defn} 

\noindent and we showed that, in particular, the following continuous functions belong to $\mathcal{G}$ (see Section 5 of \cite{bfs} for more information, properties and examples about this class) : 

\smallskip

\begin{enumerate}
\item Uniformly continuous functions on $\R^{N-1}$ 
(not necessarily globally Lipschitz-continuous) belong to $\mathcal{G}$.
 
\smallskip

\item For $ N \geq 3$. 
Any function  $g \in C^2(\R^{N-1})$, bounded from below and such that $ \nabla^{2} g\in L^{\infty}(\R^{N-1})$ is a member of the class  $\mathcal{G}. $
\\ 
For instance, $ g= g(x_1,\ldots, x_{N-1}) = {(x_1)}^2 + \prod_{j=2}^{N-1} \sin(j x_j)$ belongs to $\mathcal{G}$. 

\smallskip

\item For $N=2$, any continuous function $g :\R \to \R$ such that ${\ell}_{-}:= \lim_{x \mapsto -\infty}g(x)\in (-\infty , + \infty]$ and $ {\ell}_{+}:= \lim_{x \mapsto +\infty} g(x) \in (-\infty , + \infty]$ belongs to $\mathcal{G}$. Therefore, any quasiconvex (resp. quasiconcave) continuous function bounded from below belongs to $\mathcal{G}$. In particular, \textit{any monotone  continuous function bounded from below} belongs to $\mathcal{G}$ as well as \textit{any convex functions bounded from below}. 

\smallskip

\item Further explicit examples of member of $\mathcal{G}$ are provided by the following functions (see Section 5 of \cite{bfs}) :  $g(x_1) = e^{x_1}$ if $ N=2$ and $g(x)=e^{x_1 + \sum_{j=2}^{N-1} \cos^j(x_j)}$ if $ N \geq 3$. Also, 
$ g(x_1,\ldots, x_{N}) = e^{e^{x_1}}$ and $g(x_1,\ldots, x_{N}) = x_1^4 + e^{x_2}$ qualify for any $ N \geq 2$. 
\end{enumerate}

\medskip


The examples above clearly show that the following result (which constitutes the main result of this section) extends the known results mentioned in the Introduction, when $f$ is a specified Allen-Cahn type nonlinearity (following the two research lines described above).

\smallskip

\begin{thm}\label{th-allen-cahn 2}
Assume $N \geq 2$ and let $\Omega $ be an epigraph bounded from below and defined by a function $g \in \mathcal{G} \cap C^1(\R^{N-1})$. \\ 
Let $f \in C^1([0, + \infty))$ be a specified Allen-Cahn type nonlinearity and let $u \in C^1(\overline{\Omega}) \cap C^2(\Omega)$ be a solution to \eqref{probleme}.

\noindent Then, $\Omega=\R^{N}_{+}$ up to a vertical translation and there exists $u_{0}:[0,+\infty) \to [0,+\infty)$ strictly increasing such that
	\begin{equation}
		u(x)=u_{0}(x_{N}) \qquad \forall x \in \R^{N}_{+}. 
	\end{equation}
\end{thm}

\begin{proof}
On the one hand, condition $\mathtt{W}2$) implies that $ \lim_{t\rightarrow 0^+} \frac{f(t)}{t}>0$. On the other hand, $u$ is bounded by Corollary \ref{cor_K-O-B-type}.  Hence, $ \nabla u$ is bounded too, thanks to Corollary \ref{cor_gradient}.
We can therefore apply Theorem 5.2. of \cite{bfs} to get that $\frac{\partial u}{\partial x_N}>0$ on $\Omega.$ Since $ \Omega$ satisfies condition a), being bounded from below, the desired conclusion then follows by applying Theorem \ref{th-allen-cahn 1}. \end{proof}

\bigskip

\noindent We conclude this section with the proof of Theorem \ref{th-allen-cahn 1}. 

\medskip

\noindent \textit{Proof of Theorem \ref{th-allen-cahn 1}.} To establish the desired conclusion we shall use some results proved by  K. Wang and J. Wei in \cite{ww} together with a blow-down procedure reminiscent from the analysis  for sets of least perimeter in the theory developed by E. De Giorgi.   

We first observe that, up to a translation of $ \Omega$ in the vertical direction, we may and do suppose that the origin 
$0$ of $\R^N$ belongs to $\partial \Omega$. Then, following \cite{ww}, we recall that any solution to \eqref{probleme} 
satisfying the monotonicity assumption \eqref{hyp-monot} (when extended by $0$ outside $ \Omega$) is a (local) minimizer of the functional 
$$ 
\int \frac{\vert \nabla u \vert^2}{2} + W(u) \textbf{1}_{\left\lbrace u>0\right\rbrace}
$$
in $ \R^N,$ which also satisfies $ 0 \leq u < 1$. 
Consequently, for a subsequence $\varepsilon_n \longrightarrow 0^+,$ the blowing down sequence 
$u_n(x) := u(\frac{x}{\varepsilon_n})$ converges in $L^1_{loc}(\R^N)$ to the characteristic function 
$\textbf{1}_{\Omega_\infty}$, where $\Omega_\infty$ is a non-trivial set of least perimeter in $\R^N$ (i.e., a set of least perimeter with $0 \in \partial \Omega_\infty$). Furthermore, on any connected compact set of $\R^N \setminus \partial \Omega_\infty$, either $u_n \longrightarrow 1$ uniformly or  $u_n \equiv 0$ for all large $n$. 

\medskip

At this point, for any $x \in \R^N$, we set $v(x) = u(x',-x_N)$, therefore $v_n(x) = u_n(x',-x_N)$,
$$
\left\lbrace v>0\right\rbrace = \left\lbrace (x',x_N) \in \R^N \, : \, x_N < -g(x') \right\rbrace, 
$$
$$
\left\lbrace v_n>0\right\rbrace = \left\lbrace (x',x_N) \in \R^N \, : \, x_N < \varphi_n(x')  \right\rbrace := U_n
$$
where 
$$ \varphi_n(x') = - \varepsilon_n g\left( \frac{x'}{\varepsilon_n}\right) \qquad x' \in \R^{N-1}.$$
Also observe that $v_n$ converges in $L^1_{loc}(\R^N)$ to $\textbf{1}_{- \Omega_\infty}$ and that on any connected compact set of $\R^N \setminus \partial (-\Omega_\infty)$, either $v_n \longrightarrow 1$ uniformly or  $v_n \equiv 0$ for all large $n$. Hence
$$
\textbf{1}_{U_N} \longrightarrow \textbf{1}_{- \Omega_\infty} \quad \text{in   } L^1_{loc}(\R^N). 
$$
The latter and Lemma 16.3 in \cite{Giu} tell us that the set $- \Omega_\infty$ is itself the subgraph of a mesurable function 
$\varphi_{\infty} : \R^{N-1} \mapsto [- \infty, + \infty]$, which is the a.e.-limit of $\varphi_n$ (up to a subsequence). 

\smallskip

To obtain the conclusion of the Theorem, it is enough to prove that the blowing down limit $\Omega_\infty$
 is a half-space (see Theorem 1.6 in \cite{ww}). To this end, below we perform the classical blow-down analysis for sets of least perimeter (see for instance the book of E. Giusti \cite{Giu}). Specifically, since $\Omega_\infty$ is a non-trivial set of least perimeter in $\R^N$, the blow-down sequence $ \frac{\Omega_\infty}{n}$ converges in $L^1_{loc}(\R^N)$ to a non-trivial minimal cone $\mathcal{C}$. We also recall that this procedure tell us that, the limit cone  $\mathcal{C}$ is a half-space if and only if $\Omega_\infty$ is a half-space (see e.g. Theorem 17.3 and Theorem 9.3 \cite{Giu}). 
 
\smallskip 
 
Hence, to conclude the proof of the Theorem it is enough to prove that $\mathcal{C}$ is a half- space. To this end, we first 
claim that 
\begin{equation}\label{cono-inv-cilindro}
\mathcal{C} + e_N \subseteq \mathcal{C},
\end{equation} 
 where $e_N$ denotes the unit vector $(0, \cdots, 1) \in \R^N$. 
Indeed, as $\Omega_\infty$ is a subgraph, it follows that $\Omega_\infty + t e_N \subseteq \Omega_\infty$ for every $t>0$. This yields that $ \frac{\Omega_\infty}{n} + e_N \subseteq \frac{\Omega_\infty}{n}$ for every $n$. 
Hence, by $L^1_{loc}(\R^N)$ convergence, $ \mathcal{C} + e_N \subseteq \mathcal{C}$. 
By \eqref{cono-inv-cilindro} and Lemma 2.3 of \cite{GMM} (cf. also Proposition 2.1 of  
\cite{farBer}) we see that either $\mathcal{C}$ is a half-space or it is a cylinder of the form $\mathcal{C} = \mathcal{C'} \times \R$, for some non-trivial minimal cone $ \mathcal{C'} \subseteq \R^{N-1}$.

When $ N-1 \leq 7$, all the non-trivial minimal cones are half-spaces (see \cite{Giu}), hence for $ N \leq 8$ we have that $\mathcal{C}$ is a half-space and so the Theorem is proved. 

Now we turn to the case $N \geq 9$ and we first suppose that $g$ satisfies the assumption $a)$. 
Under this assumption we see that $\R^N \setminus \mathcal{C}$ contains the cone
$$
\mathfrak{C} := \left\lbrace (x',x_N) \in \R^N \, : \, x_N > -C \vert x' \vert \right\rbrace .
$$ 
Indeed, for almost every $ x' \in \R^{N-1}$, 
$$
\varphi_{\infty}(x')\longleftarrow \varphi_n (x') \leq  \varepsilon_n C \left( 1 + \Big \vert \frac{x'}{ \varepsilon_n} \Big \vert \right) \longrightarrow C \vert x'\vert,
$$
that is
\begin{equation}\label{stima-cono}
\varphi_{\infty}(x') \leq C \vert x' \vert \quad \text{for almost every } x' \in \R^{N-1}, 
\end{equation}
which immediately implies that $\mathfrak{C} \subseteq \R^N \setminus \Omega_\infty$. 

Also, since $- \Omega_\infty$ is the subgraph of $\varphi_{\infty}$, we see that $\textbf{1}_{- \frac{\Omega_\infty}{n}} \longrightarrow \textbf{1}_{- C} \quad \text{in   } L^1_{loc}(\R^N)$. Therefore, $-C$ is itself the subgraph of a mesurable function $\psi_{\infty} : \R^{N-1} \mapsto [- \infty, + \infty]$, which is the a.e.-limit of $(\varphi_{\infty})_n$ (up to a subsequence) defined by 
$$
(\varphi_{\infty})_n(x') = \varepsilon_n \varphi_{\infty} \left( \frac{x'}{\varepsilon_n}\right) \qquad x' \in \R^{N-1}.
$$
The latter and \eqref{stima-cono} clearly imply that $\mathfrak{C} \subseteq \R^N \setminus \mathcal{C}$ . 

\smallskip

If $\mathcal{C}$ is not a half-space, then $\mathcal{C} = \mathcal{C'} \times \R$, for some non-trivial minimal cone $ \mathcal{C'} \subseteq \R^{N-1}$. The latter and the inclusion $\mathfrak{C} \subseteq \R^N \setminus \mathcal{C}$ imply that 
$ \mathcal{C'} = \R^{N-1}$, contradicting its non-triviality.   This concludes the proof when $g$ satisfies the assumption a). A similar argument provides the desired result under condition b). \qed	

\bigskip

\section{Serrin's overdetermined problem with general nonlinearity. \\The case $f(0) \geq 0$. } \label{Sect-f(0)non-neg}

In this section we consider a general locally Lipschitz-continuous function $f$ satisfying $f(0)\geq 0$ and solutions $u$, possibly unbounded, to the overdetermined problem \eqref{probleme}. In this framework we prove the rigidity of $\Omega$ and the one-dimensional symmetry of $u$ in dimension $N \leq 3$. As already discussed in the Introduction, the proofs are obtained by combining the monotonicity results recently proved in \cite{bfs} together with the geometric approach developed in \cite{FVarma}. Indeed, the following general rigidity result for monotone solutions to the overdetermined problem \eqref{probleme} holds true when $\Omega$ is a domain of $\R^2$ or $\R^3$.  

\medskip

\begin{thm}\label{Teo1-lip-gen} Assume $ N=2,3$ and let $\Omega \subset \R^N$ be a domain of class $C^3$. \\
Let $f \in {Lip}_{loc}([0,+\infty))$ and let $u \in C^2(\overline{\Omega})$ be a  solution to \eqref{probleme} such that 
	
\begin{equation}\label{crescita-energia-gen}
		\disp\int_{B(0,R)\cap \Omega}|\nabla u|^{2} = o(R^2 \ln R) \qquad \textit{as } \, R \longrightarrow \infty.
\end{equation}  	
If $u$ is monotone, i.e., 
\begin{equation}\label{hyp-monot-N}
		\frac{\partial u}{\partial x_N}(x)>0 \qquad \forall x \in \Omega,
\end{equation}
then, $\Omega=\R^N_+$ up to isometry and there exists $u_{0}:[0,+\infty) \to (0,+\infty)$ strictly increasing such that
	\begin{equation*}
		u(x)=u_{0}(x_N) \quad \forall x \in \R^N_{+}.
	\end{equation*}
\end{thm}

\smallskip

\begin{rem}
Note explicitly that in the above Theorem :

\smallskip

\noindent (i) no restriction is imposed on the sign of $f(0)$ or on $\mathfrak{c}$; 

\noindent (ii) $\Omega$ is not assumed to be an epigraph. 
\end{rem}

\smallskip

\begin{proof} We use the geometric approach developped in \cite{FVarma}. To this end we first recall that, given a smooth function $v$, one may consider the level sets $\{v=c\}$. In the vicinity of $\{\nabla v \neq 0\}$, these level sets are smooth manifolds, so one can introduce the principal curvatures
\begin{equation*}
	k_{1}, \cdots , k_{N-1}
\end{equation*} 
at any point of such manifolds.
We set 
\begin{equation*}
	\mathcal{K}:=\sqrt{k_{1}^{2}+\cdots+k_{N-1}^{2}}.
\end{equation*}
Also, we consider the tangential gradient along level sets of $v$ at these points, that is
\begin{equation*}
	\nabla_{T}v := \nabla v -\Big( \nabla v \cdot  \frac{\nabla v}{|\nabla v|}\Big)\frac{\nabla v}{|\nabla v|} .
\end{equation*}
With these notations, and given that $u$ is monotone with respect to the vertical direction, we can apply Theorem 1.1 of \cite{FVarma} to obtain 
	\begin{equation}\label{inegalite_poincare}
		\displaystyle\int_{\Omega} (|\nabla u|^{2} \mathcal{K}^{2}+|\nabla_{T}|\nabla u||^{2})\phi^{2} \leq \displaystyle\int_{\Omega}|\nabla u|^{2}|\nabla \phi|^{2}, 
	\end{equation}
	for any compactly supported, Lipschitz-continuous function $ \phi : \R^N \mapsto \R$.

\smallskip

Now, for any $R\geq 1$ and $x \in \R^N$ we define 
	\begin{equation*}
		\phi_{R}(x):=\textbf{1}_{B_{\sqrt{R}}}(x)+\dfrac{2\ln(R/|x|)}{\ln R}\textbf{1}_{B_{R}\backslash B_{\sqrt{R}}}(x).
	\end{equation*}
Then, $\phi_{R}$ can be inserted into \eqref{inegalite_poincare} to get 
	\begin{equation}\label{stima-curvature}
		\disp\int_{\Omega \cap B_{\sqrt{R}}}|\nabla u|^{2}\mathcal{K}^{2}+ |\nabla_{T}|\nabla u||^{2}\leq \frac{4}{\ln^2R }\disp\int_{(B_{R}\backslash B_{\sqrt{R}}) \cap \Omega}\frac{|\nabla u(x)|^{2}}{|x|^{2}}dx.
	\end{equation}
	Moreover, since  $\disp\int_{|x|}^{R}\frac{2}{t^{3}}dt=\frac{1}{|x|^{2}}-\frac{1}{R^{2}}$, we have  
	\begin{equation*}
		\begin{split}
			\disp\int_{(B_{R}\backslash B_{\sqrt{R}}) \cap \Omega} \frac{|\nabla u(x)|^{2}}{|x|^{2}}&\leq \disp\int_{(B_{R}\backslash B_{\sqrt{R}}) \cap \Omega} \disp\int_{|x|}^{R}\frac{2|\nabla u(x)|^{2}}{t^{3}}dtdx+\disp\int_{(B_{R}\backslash B_{\sqrt{R}}) \cap \Omega} \frac{|\nabla u|^{2}}{R^{2}}\\
			&\leq \disp\int_{\sqrt{R}}^{R}\disp\int_{B_{t}\backslash B_{\sqrt{R}}\cap \Omega} \frac{2|\nabla u(x)|^{2}}{t^{3}}dxdt +\disp\int_{(B_{R}\backslash B_{\sqrt{R}}) \cap \Omega} \frac{|\nabla u|^{2}}{R^{2}}\\
& \leq \disp\int_{\sqrt{R}}^{R} \frac{2}{t^3} \left[ \int_{B_{t}\backslash B_{\sqrt{R}}\cap \Omega} |\nabla u(x)|^{2} dx\right] dt  
+ \frac{1}{R^2}  \int_{(B_{R}\backslash B_{\sqrt{R}}) \cap \Omega} |\nabla u |^{2} \\ 
& = o(\ln^2 R) \qquad \textit{as } \, R \longrightarrow \infty, \\			
		\end{split},
	\end{equation*}
where in the latter line we have used the assumption \eqref{crescita-energia-gen}. Using this information, and letting 
$ R \longrightarrow \infty$ in \eqref{stima-curvature}, we obtain that $\mathcal{K}^{2} + |\nabla_{T}|\nabla u||^{2} \equiv 0$ 
in $\Omega$. The latter implies that $\Omega$ is an affine half-space and $u$ is one-dimensional (see the level set analysis in Section 2.1. of \cite{FVarma}). This proves the result. \end{proof}

\medskip
The results presented below are entirely new in the three-dimensional case, even for bounded solutions. When $N=2$, they are new for unbounded solutions, since the case of bounded solutions is covered by the results proved in \cite{RRS17} (as already mentioned in the Introduction).

\begin{thm}\label{Teo2-lip} Let $\Omega \subset \R^2$ be a globally Lipschitz-continuous epigraph bounded from below and with boundary of class $ C^3$ and let $u \in C^1(\overline{\Omega}) \cap C^2(\Omega)$ be a  solution to \eqref{probleme}. 

\noindent Assume that one of the following hypotheses holds true :

	\begin{align*}
		& \tag{H1} f \in {Lip}_{loc}([0,+\infty)), f(0)\geq 0 \text{ and }  \nabla u \in L^{\infty}(\Omega);
		\\[4pt]
		& \tag{H2}\left\{
		\begin{array}{ccc}
			f \in {Lip}_{loc}([0,+\infty)), \quad f(0)\geq 0, \quad \exists \, \zeta > 0 \, : \, f(t) \leq 0  \quad \textit{in}  \quad [\zeta,+\infty),\\
			\\
			\exists \, \alpha \in [0,1) \, : \quad  u(x)= O(\vert x \vert^{\alpha}) \qquad  \textit{as}  \quad \vert x \vert \longrightarrow \infty; \hfill
		\end{array}
		\right.
		\\[4pt]
		& \tag{H3}\left\{
		\begin{array}{ccc}
			f \in {Lip}([0,+\infty)), \quad f(0)\geq 0, \quad \exists \, \zeta > 0 \, : \, f(t) \geq 0  \quad \textit{in}  \quad [\zeta,+\infty),\\
			\\
			u(x)= o (\ln \vert x \vert), \qquad  \textit{as}  \quad \vert x \vert  \longrightarrow \infty; \hfill
		\end{array}
		\right.
		\\[4pt]
		& \tag{H4}\left\{ \begin{array}{ccc}
			f \in {Lip}([0,+\infty)), \quad f(0)= 0, \quad f(t) \leq 0  \quad \textit{in}  \quad(0,+\infty),\\
			\\
			u(x) = o(\vert x \vert  \ln^{\frac{1}{2}} \vert x \vert), \qquad  \textit{as}  \quad \vert x \vert \longrightarrow \infty. \hfill
		\end{array}
		\right.
	\end{align*}
\bigskip

Then, $\Omega=\R^2_+$ up to a vertical translation and there exists $u_{0}:[0,+\infty) \to (0,+\infty)$ strictly increasing such that
	\begin{equation*}
		u(x)=u_{0}(x_2) \quad \forall x \in \R^2_{+}.
	\end{equation*}
\end{thm}

\smallskip

\begin{proof} Assume (H1).  By Corollary 1.5 of \cite{bfs} $u$ is monotone, i.e., $\frac{\partial u}{\partial x_2}>0$ in 
$\Omega$. Furthermore, we have 
$$
\int_{B(0,R)\cap \Omega}|\nabla u|^{2}  \leq \Vert \nabla u \Vert^2_{L^\infty(\Omega)} \pi R^2 \qquad R>0,
$$
The conclusion then follows from Theorem \ref{Teo1-lip-gen}. 

Assume (H2).  By Proposition \ref{prop_princ-maxBCN} $u$ is bounded. Then, also $\nabla u$ is bounded (see Corollary \ref {cor_gradient}).  We can apply the previous step to conclude. 

Assume (H3).  The growth assumption on $u$ implies that $u$ has at most exponential growth on any finite strip $\Omega \cap \left \lbrace x_N <R \right\rbrace$. Therefore, we can apply Theorem 1.3 of \cite{bfs} to obtain that $\frac{\partial u}{\partial x_2}>0$ in $\Omega$. To get the desired result, all we need to do is to prove that  \eqref{crescita-energia-gen} holds true. To this end, for $t>0$ we set $ M(t) := \sup_{B(0,t)\cap \Omega} u$ and then  we pick a function $\varphi \in C^{1}_c (\R^2)$ such that $ 0 \le \varphi \le 1,$ 
\begin{equation*}\label{cut-off}
\varphi(x) := \begin{cases}
1\quad & {\rm if} \qquad \vert x \vert \le 1,\\
0 & {\rm if} \qquad \vert x \vert \ge 2,
\end{cases}
\end{equation*}
and, for any $R > 0 $ and $x \in \R^2$ we set $\varphi_R(x) = \varphi({x \over {R}})$. 

Then we multiply the equation by 
$(u-M(2R)) \varphi_R^2$ and we integrate by parts to find 
\begin{equation*}
		\begin{split}
&\int_{B_{2R} \cap \Omega}|\nabla u|^{2} \varphi_R^2 + \int_{B_{2R} \cap \Omega} 2 \varphi_R(u-M(2R)) \nabla u \nabla \varphi_R  - \int_{B_{2R} \cap \partial \Omega} \frac{\partial u}{\partial \eta} (u-M(2R)) \varphi_R^2 \\
& = \int_{B_{2R} \cap \Omega} f(u) (u-M(2R)) \varphi_R^2 \leq \int_{(B_{2R} \cap \Omega) \cap \left\lbrace u < \zeta \right\rbrace } f(u) (u-M(2R)) \varphi_R^2 \\
& \leq 4 \omega_2 \Vert f \Vert_{L^{\infty}([0, \zeta])} M(2R) R^2. 
		\end{split}
\end{equation*}
Thus
\begin{equation*}
		\begin{split}
\int_{B_{2R} \cap \Omega}|\nabla u|^{2} \varphi_R^2 & \leq - \int_{B_{2R} \cap \Omega} 2 \varphi_R(u-M(2R)) \nabla u \nabla \varphi_R  \\
&+ \vert \mathfrak{c}\vert M(2R) \int_{B_{2R} \cap \partial \Omega} \varphi_R^2 + 4 \omega_2 \Vert f \Vert_{L^{\infty}([0, \zeta])} M(2R) R^2 \\
& \leq \frac{1}{2} \int_{B_{2R} \cap \Omega} |\nabla u|^{2} \varphi_R^2 + 2  \int_{B_{2R} \cap \Omega} (u-M(2R))^2 \vert \nabla 
\varphi_R \vert^2   \\
& + \vert \mathfrak{c} \vert M(2R) \mathcal{H}^1(B_{2R} \cap \partial \Omega) +  4 \omega_2 \Vert f \Vert_{L^{\infty}([0, \zeta])} M(2R) R^2 \\
		\end{split}
\end{equation*}
and so 
\begin{equation}\label{quasi-prima-stima-energia-lip}
		\begin{split}
\int_{B_{R} \cap \Omega}|\nabla u|^{2} & \leq 4 \int_{B_{2R} \cap \Omega} (M(2R)-u)^2 \vert \nabla \varphi_R \vert^2  \\
& + 2 \vert \mathfrak{c} \vert M(2R) \mathcal{H}^1(B_{2R} \cap \partial \Omega) +  8 \omega_2 \Vert f \Vert_{L^{\infty}([0, \zeta])} M(2R) R^2 \\
& \leq 4 M^2(2R) \int_{B_{2R} \cap \Omega} \vert \nabla \varphi_R \vert^2  + 2 \vert \mathfrak{c} \vert M(2R) 
\mathcal{H}^1(B_{2R} \cap \partial \Omega)  \\
& +  8 \omega_2 \Vert f \Vert_{L^{\infty}([0, \zeta])} M(2R) R^2  \leq 4 M^2(2R) \Vert \nabla \varphi \Vert^2_{L^{\infty}(\R^2)} R^{-2} \omega_2 (2R)^2 \\ 
& + 2 \vert \mathfrak{c} \vert M(2R) \mathcal{H}^1(B_{2R} \cap \partial \Omega)  +  8 \omega_2 \Vert f \Vert_{L^{\infty}([0, \zeta])} M(2R) R^2 \\
& \leq 16 \omega_2 \Vert \nabla \varphi \Vert^2_{L^{\infty}(\R^2)} M^2(2R) +  2 \vert \mathfrak{c}\vert M(2R)   
\mathcal{H}^{1} (B_{2R} \cap \partial \Omega)\\
& +  8 \omega_2 \Vert f \Vert_{L^{\infty}([0, \zeta])} M(2R) R^2,
		\end{split}
\end{equation}
where $\mathcal{H}^{1}$ denotes the 1-dimensional Hausdorff measure.

We also observe that 
\begin{equation*}
\mathcal{H}^{1} (B_{2R} \cap \partial \Omega) \leq \int_{(-2R,2R) \subset \R} \sqrt{1 + \vert g'(x_1) \vert^2} dx_1 
\leq 4 (1 + \Vert g' \Vert_{L^{\infty}(\R)} ) R, 
\end{equation*}
since $g$ is globally Lipschitz-continuous by assumption. Then, from \eqref{quasi-prima-stima-energia-lip} and the growth assumption on $u$, we deduce that 
$$
\int_{B(0,R)\cap \Omega}|\nabla u|^{2} = o(R^2 \ln R) \qquad \textit{as } \, R \longrightarrow \infty.
$$ 

Assume (H4). In view of the growth assumption on $u$ we can apply Theorem 1.3 of \cite{bfs} to obtain that $\frac{\partial u}{\partial x_2}>0$ in $\Omega$. To conclude, we prove that \eqref{crescita-energia-gen} holds true. To this purpose we multiply the equation by $u \varphi_R^2$ and we integrate by parts to find 
\begin{equation*}
		\begin{split}
&\int_{B_{2R} \cap \Omega}|\nabla u|^{2} \varphi_R^2 + \int_{B_{2R} \cap \Omega} 2 \varphi_R u \nabla u \nabla \varphi_R  - \int_{B_{2R} \cap \partial \Omega} \frac{\partial u}{\partial \eta} u \varphi_R^2 = \int_{B_{2R} \cap \Omega} f(u) u \varphi_R^2 \leq 0,
		\end{split}
\end{equation*}
hence
\begin{equation*}
		\begin{split}
\int_{B_{2R} \cap \Omega}|\nabla u|^{2} \varphi_R^2 & \leq  - \int_{B_{2R} \cap \Omega} 2 \varphi_R u \nabla u \nabla \varphi_R  + \int_{B_{2R} \cap \partial \Omega} \frac{\partial u}{\partial \eta} u \varphi_R^2 \\
& = - \int_{B_{2R} \cap \Omega} 2 \varphi_R u \nabla u \nabla \varphi_R +0 \\
& \leq \frac{1}{2} \int_{B_{2R} \cap \Omega} |\nabla u|^{2} \varphi_R^2 + 2  \int_{B_{2R} \cap \Omega} u^2 \vert \nabla 
\varphi_R \vert^2.
		\end{split}
\end{equation*}
Therefore, for any $R>0$, 
\begin{equation*}
		\begin{split}
\int_{B_{R} \cap \Omega}|\nabla u|^{2} \leq  \int_{B_{2R} \cap \Omega}|\nabla u|^{2} \varphi_R^2  \leq 
4 \int_{B_{2R} \cap \Omega} u^2 \vert \nabla  \varphi_R \vert^2 \leq 16 \omega_2 \Vert \nabla  \varphi \Vert^2_{L^{\infty}(\R^2)}   M^2(2R) 
		\end{split}
\end{equation*}
and the condition  \eqref{crescita-energia-gen}  then follows. \end{proof}

\medskip

\begin{thm}\label{Teo3-lip} Let $\Omega \subset \R^3$ be a globally Lipschitz-continuous epigraph bounded from below and with boundary of class $C^3$ and let $u \in C^1(\overline{\Omega}) \cap C^2(\Omega)$ be a  solution to \eqref{probleme}. 

\noindent Assume that one of the following assumptions holds true :
	
\smallskip

 \begin{equation}\tag{A1}
	\left\{
	\begin{array}{ccc}
		f \in {Lip}([0,+\infty)), \quad f(t) \geq 0  \quad \textit{for any}  \quad t \geq 0, \\
		\\
		u(x)= o (\ln \vert x \vert), \qquad  \textit{as}  \quad \vert x \vert  \longrightarrow \infty; \hfill
	\end{array}
	\right.
\end{equation}
\begin{equation}\tag{A2}
	\left\{
	\begin{array}{ccc}
		f \in {Lip}_{loc}([0,+\infty)), \quad \exists \, \zeta \geq 0 \, : \, f(t) \geq 0  \quad \textit{on} \quad [0,\zeta] \quad\textit{and} \quad
		f(t)\leq 0 \quad \textit{on} \quad [\zeta,+\infty), \\
		\\
		\exists \, \alpha \in [0,1) \, : \quad  u(x)= O(\vert x \vert^{\alpha}) \qquad  \textit{as}  \quad \vert x \vert \longrightarrow \infty; \hfill
	\end{array}
	\right.
\end{equation}
\bigskip

Then, $\Omega=\R^3_+$ up to a vertical translation and there exists $u_{0}:[0,+\infty) \to (0,+\infty)$ strictly increasing such that
	\begin{equation*}
		u(x)=u_{0}(x_3) \quad \forall x \in \R^3_{+}.
	\end{equation*}
\end{thm}

\medskip

\begin{proof} 
Assume (A1). The proof follows the same lines as the proof of the previous Theorem under the assumption (H3). First, 
by Theorem 1.3 of \cite{bfs} we obtain that $\frac{\partial u}{\partial x_3}>0$ in $\Omega$. Then we multiply the equation by 
$(u-M(2R)) \varphi_R^2$, where now $ \varphi \in C^3_c(\R^3)$, and we integrate by parts to find 
\begin{equation*}
		\begin{split}
&\int_{B_{2R} \cap \Omega}|\nabla u|^{2} \varphi_R^2 + \int_{B_{2R} \cap \Omega} 2 \varphi_R(u-M(2R)) \nabla u \nabla \varphi_R  - \int_{B_{2R} \cap \partial \Omega} \frac{\partial u}{\partial \eta} (u-M(2R)) \varphi_R^2 \\
& = \int_{B_{2R} \cap \Omega} f(u) (u-M(2R)) \varphi_R^2  \leq 0,
		\end{split}
\end{equation*}
and so
\begin{equation}
		\begin{split}
\int_{B_{R} \cap \Omega}|\nabla u|^{2} & \leq 4 \int_{B_{2R} \cap \Omega} (M(2R)-u)^2 \vert \nabla \varphi_R \vert^2  + 2 \vert \mathfrak{c} \vert M(2R) \mathcal{H}^2(B_{2R} \cap \partial \Omega) \\
& \leq 32 \omega_3 \Vert \nabla \varphi \Vert^2_{L^{\infty}(\R^3)} M^2(2R) R
 + 2 \vert \mathfrak{c} \vert M(2R) \mathcal{H}^2 (B_{2R} \cap \partial \Omega). 
		\end{split}
\end{equation}
where $\omega_3$ is the measure of the 3-dimensional unit ball and $\mathcal{H}^{2}$ denotes the 2-dimensional Hausdorff measure.

On the other hand, 
\begin{equation*}
\mathcal{H}^{2} (B_{2R} \cap \partial \Omega) \leq \int_{B'(0',2R) \subset \R^2} \sqrt{1 + \vert \nabla^{'} g(x') \vert^2} dx' 
\leq 4 \pi (1 + \Vert \nabla^{'} g \Vert_{L^{\infty}(\R^2)}  ) R^2, 
\end{equation*}
since $g$ is globally Lipschitz-continuous by assumption. Consequently, for any $R>0$, we have 
\begin{equation*}
		\begin{split}
\int_{B_{R} \cap \Omega}|\nabla u|^{2} & \leq 32 \omega_3 \Vert \nabla \varphi \Vert^2_{L^{\infty}(\R^3)} M^2(2R) R
 + 2 \vert \mathfrak{c} \vert M(2R) \mathcal{H}^2 (B_{2R} \cap \partial \Omega)  \\
& \leq 32 \omega_3 \Vert \nabla \varphi \Vert^2_{L^{\infty}(\R^3)}  M^2(2R) R
 + 8 \vert \mathfrak{c} \vert \pi (1 + \Vert \nabla^{'} g \Vert_{L^{\infty}(\R^2)}  ) M(2R) R^2. 
		\end{split}
\end{equation*}
The latter and the assumption on $u$ imply the validity of the condition \eqref{crescita-energia-gen}.  Then, the 
conclusion follows by applying Theorem \ref{Teo1-lip-gen}.

Assume (A2).   By Proposition \ref{prop_princ-maxBCN} $u$ is bounded above by $\zeta$.  Hence, $u$ is a bounded solution to 
\begin{equation*}
\left\{
	\begin{array}{cll}
		- \Delta u=\tilde{f}(u) & \text{in}& \Omega,\\
		u > 0& \text{in} & \Omega,\\
		u=0 & \text{on} & \dr\Omega,\\
		\dfrac{\partial u}{\partial \eta} = \mathfrak{c} = const. & \text{on} & \dr \Omega,
	\end{array}
	\right.
\end{equation*}	
where $\tilde{f} $ is the non-negative globally Lipschitz-continuous function defined by $\tilde{f}=f$ in $[0, \zeta]$ and 
$\tilde{f} \equiv f(\zeta)=0 $ in $[\zeta, + \infty)$. Therefore, we can apply the previous step to conclude. \end{proof} 

\bigskip

When $f$ is of class $C^1$ we can extend the above result to a larger class of nonlinear functions $f$, with $ f(0) \geq 0$. Specifically, we have 

\medskip

\begin{thm}\label{Teo1-C1}
Let $\Omega\subset \R^3$ be a globally Lipschitz-continuous epigraph bounded from below and with boundary of class 
$C^3$. \\
Let $u \in C^1(\overline{\Omega}) \cap C^2(\Omega)$ be a  bounded solution to \eqref{probleme}, where $f \in C^1([0,+\infty))$ satisfies $f(0) \geq 0$ and 
\begin{equation}\label{cond-$c_u$}
		F( \sup u) = \sup_{t \in [0, \sup u]} F(t),
\end{equation}
here $F$ denotes the primitive of $f$ vanishing at $0$.\footnote{ $ \, F(t)=\disp\int_{0}^{t}f(s)ds, \qquad t \geq 0.$}

\noindent  Then, $\Omega=\R^3_{+}$ up to a vertical translation and there exists $u_{0}:[0,+\infty) \to (0,+\infty)$ strictly increasing such that
	\begin{equation*}
		u(x)=u_{0}(x_{3}) \quad \forall x \in \R^3_{+}.
	\end{equation*}
\end{thm}

\medskip

\begin{rem}
Note that the assumption \eqref{cond-$c_u$} is natural and somehow necessary. Indeed, if $\Omega$ is an affine half-space, then the Unique Continuation Principle ensures that the solution $u$ must be one-dimensional. Next, a standard analysis of the first integral gives that $F(\sup u) > F(t)$, for any $t \in [0,\sup u)$, thus confirming the validity of  \eqref{cond-$c_u$}. 
\end{rem}

\bigskip

The proof of Theorem \ref{Teo1-C1} is based on the following general rigidity result for monotone solutions to the overdetermined problem \eqref{probleme} on 3-dimensional epigraphs.  

\medskip

\begin{thm}\label{Teo1-gen}
Let $u \in C^2(\overline{\Omega})$ be a  bounded solution to \eqref{probleme} where $\Omega\subset \R^3$ is an epigraph with boundary of class $C^3$ such that 
	\begin{equation}\label{measure-bord}
    \mathcal{H}^{2} (\partial \Omega \cap B(0,R)) = o(R^2 \ln(R)) \qquad \textit{as } \, R \longrightarrow \infty,
	\end{equation}
where $\mathcal{H}^{2}$ denotes the 2-dimensional Hausdorff measure. 
	
Let $f \in C^{1}([0,+\infty))$ be such that 
\begin{equation}
		F( \sup u) = \sup_{t \in [0, \sup u]} F(t).
\end{equation}

If $u$ is monotone, i.e., 
\begin{equation}\label{hyp-monot-3}
		\frac{\partial u}{\partial x_3}(x)>0 \qquad \forall x \in \Omega,
\end{equation}
then, $\Omega=\R^3_+$ up to isometry and  there exists $u_{0}:[0,+\infty) \to (0,+\infty)$ strictly increasing such that
	\begin{equation*}
		u(x)=u_{0}(x_{3}) \quad \forall x \in \R^3_{+}.
	\end{equation*}
\end{thm}

\smallskip

\begin{rem}
Note explicitly that in the above Theorem :

\noindent (i) no restriction is imposed on the sign of $f(0)$ or on $\mathfrak{c}$; 

\noindent (ii) the epigraph $\Omega$ is not assumed to be bounded from below. 
\end{rem}

\bigskip

\noindent \textit{Proof of Theorem \ref{Teo1-gen}.} 

From Lemma \ref{lem-energy-est-3-d} and assumption 
\eqref {measure-bord} we have 
\begin{equation}\label{stima-energia-3-dim-precisata}
		\disp\int_{B(0,R)\cap \Omega}|\nabla u|^{2} = o(R^2 \log R) \qquad \textit{as } \, R \longrightarrow \infty.
\end{equation}  
The conclusion then follows by appyling Theorem \ref{Teo1-lip-gen}. \qed	

\bigskip

\noindent \textit{Proof of Theorem \ref{Teo1-C1}.} Under the assumptions of Theorem \ref{Teo1-C1}  we can apply Theorem 1.3 of \cite{bfs} to get that $\frac{\partial u}{\partial x_3}>0$ on $\Omega.$ Moreover
\begin{equation*}
\mathcal{H}^{2} (B_{R} \cap \partial \Omega) \leq \int_{B'(0',R) \subset \R^2} \sqrt{1 + \vert \nabla^{'} g(x') \vert^2} dx' 
\leq \pi (1 + \Vert \nabla^{'} g \Vert_{L^{\infty}(\R^2)} ) R^2,
\end{equation*}
since $g$ is globally Lipschitz-continuous by assumption. Hence, assumption \eqref {measure-bord} is satisfied. We conclude by applying Theorem \ref{Teo1-gen}.\qed

\bigskip

\section{Serrin's overdetermined problem with general nonlinearity. \\ The case $f(0) <0$. } \label{Sect-f(0) neg}

In this section we consider the case of locally Lipschitz-continuous functions $f$ with $f(0) <0$. 
As explained in the Introduction, we first prove a new  monotonicity result for solutions to the overdetermined problem \eqref{probleme} on epigraphs bounded from below and with non-zero Neumann boundary condition. From this result, and the geometric approach developed in \cite{FVarma}, we prove that the epigraph $\Omega $ is an affine half-space and $u$ is one-dimensional when $N \leq 3$. These results apply to solutions that may be unbounded, and they are new even in the case of bounded solutions.



\subsection{A new monotonicity result in any dimension $N \geq 2$} \quad \\

\noindent In order to state our results we need the following 

\begin{defn}
Assume $N \geq 2$ and $\gamma \in (0,1]$. We denote by $\mathbb{C}^{1,\gamma}(\R^{N-1})$ the set of functions 
$h \in C^1(\R^{N-1})$ such that $\nabla h $ is bounded and globally $\gamma$-Hölder continuous, i.e., such that 
	\begin{equation*}
		\begin{split}
		\Vert \nabla h \Vert_{C^{0,\gamma}(\R^{N-1})} &:= \|\nabla h\|_{L^{\infty}(\R^{N-1})} + \left[ \nabla h \right]_{C^{0,\gamma}  (\R^{N-1})} \\
		&:= \sup_{x \in \R^{N-1}} |\nabla h(x)| + \sup_{x,y \in \R^{N-1}, x\neq y}\frac{|\nabla h(x)-\nabla h(y)|}{|x-y|^{\gamma}}< +\infty.	
		\end{split}
	\end{equation*}
\end{defn}

\noindent Notice that, we are not assuming that $h$ is bounded on $\R^{N-1}$. However, any member of 
$\mathbb{C}^{1,\gamma}(\R^{N-1})$ is globally Lipschitz-continuous on $\R^{N-1}$, since $  \nabla h$ is supposed to be bounded on $\R^{N-1}$. 

\begin{thm}\label{th_overdet-f(0)<0}
Assume $N \geq 2$, $\gamma\in (0,1]$. Let $\Omega $ be an epigraph bounded from below and defined by a function 
$g \in \mathbb{C}^{1, \gamma}(\R^{N-1})$ and let $f \in {Lip}_{loc}(\R)$ with $f(0) <0$. \\
Let $u \in C^1(\overline{\Omega}) \cap C^2(\Omega)$ be a solution to \eqref{probleme} with $\mathfrak{c} \neq 0$ and such that $ \nabla u$ is bounded on finite strips.\footnote{\, i.e., for any $R>0$, 
\begin{equation*}
\sup_{\Omega \cap \left \lbrace x_N <R \right\rbrace } \vert \nabla u \vert \,\, < + \infty.
\end{equation*} }
\\ Then, $u$ is monotone, i.e.
	\begin{equation*}
		\frac{\partial u}{\partial x_N}(x)>0 \qquad \forall x \in \Omega . 
	\end{equation*}
\end{thm}

\smallskip

\begin{rem}
In particular, the above Theorem applies to bounded solutions to \eqref{probleme}. Indeed, bounded solutions to \eqref{probleme} also have bounded gradient (see Corollary \ref{cor_gradient}).
\end{rem}

\smallskip

\begin{proof} The proof is inspired by the method developed by the authors and B. Sciunzi in their recent paper \cite{bfs} (see Section 4 therein). However, as noted in the Introduction, the case where $f(0) < 0$ is more involved, and we do indeed require a new proof to compensate for the absence of the strong maximum principle. It is the non-zero Neumann boundary condition that allows us to circumvent this problem and prove the desired monotonicity result. The fact that we must also deal with the overdetermined Neumann condition requires us to control $u$ and all its first derivatives, which calls for a different (and more refined) proof from that the one originally devised in \cite{bfs} (see, in particular, Theorem \ref{estimé_c_1}, Proposition \ref{prop0.1} and Lemma \ref{extension_C_1-cylinder} in the Appendix).

Thanks to the translation invariance of problem \eqref{probleme}, we may and do suppose that $\disp\inf_{\rnu} g=0$.
Moreover, by the assumption on $\nabla u$, it is immediate to see that $u$ is bounded and uniformly continuous on any 
finite strip $\Omega \cap \{x_{N}<R\}$, $ R>0$. 

Therefore, as in \cite{bfs}, we can use a variant of the moving planes method {\textit {suitably adapted to the geometry of the epigraph}}. To this end, let us recall the notations used in \cite{bfs}.

\bigskip

\noindent For $0<a<b$ and  $\lambda>0$ we set: 

\smallskip

$ \Sigma_{\lambda} := \{ \,  x \in \R^N \, : \, 0 < x_N < \lambda  \, \}$,

$\Sigma_{b}^{g}=\lbrace x=(x',x_{N})\in \R^{N} \, : \,  g(x')<x_{N}<b \rbrace$,  

$\Sigma _{a,b}^{g}=\lbrace x=(x',x_{N})\in \R^{N} \, : \, g(x')+a<x_{N}<b \rbrace$,

$$
\forall x \in \Sigma_{\lambda}^{g}, \qquad u_{\lambda}(x)=u(x',2\lambda-x_{N}) .
$$

\smallskip

Also, for any subset $S \subseteq \R^N$, we denote by $UC(S)$ the set of uniformly continuous functions defined on $S$. 

\bigskip

We proceed as in the proof of Theorem $1.1$ of \cite{bfs}.

We set $\Lambda:= \lbrace t >0  \,\, : \,\, u \leqslant u_{\theta} \,\,\, \text{in} \,\,\, \Sigma_{\theta}^{g} \,, \,\,  \forall \, 
0<\theta< t \rbrace$ and we aim at proving that 
$$
\tilde{t}:=\sup \Lambda = +\infty.
$$
Exactly as in the first step of the proof of Theorem $1.1$ of \cite{bfs} we have :

\smallskip

\textit{Step 1 : $\Lambda$ is not empty.}   

\smallskip

\noindent We note explicitely that this result holds true irrespectively of the value of $f(0)$ (see also Corollary 5.6 of \cite{bfs}). 

\smallskip

Then we prove 

\smallskip

\textit{Step 2 : $\tilde{t}=\sup \Lambda = +\infty$. } 
		
\smallskip

To this aim, we establish the following 
	\begin{prop}\label{propositionimportante}
		For every $\delta \in (0,\frac{\tilde{t}}{2})$, there is $\varepsilon(\delta)>0$ such that 
		\begin{equation*} 	
			\forall \varepsilon \in (0,\varepsilon(\delta)) \quad u\leq u_{\tilde{t}+\varepsilon} \quad \text{on} \quad \overline{\Sigma_{\delta,\tilde{t}-\delta}^{g}}.
		\end{equation*}
	\end{prop}

	\noindent \textit{Proof of Proposition \ref{propositionimportante}.}
	    Suppose that there exists $\delta \in \Big(0,\frac{\tilde{t}}{2}\Big)$ such that
		\begin{equation}\label{hyp_contradiction}	
			\forall k \geq 1  \quad \exists \varepsilon_{k} \in \Big(0,\frac{1}{k}\Big) \quad \exists x^{k} \in \overline{\Sigma_{\delta,\tilde{t}-\delta}^{g}} \quad : \quad u(x^{k}) \geq u_{\tilde{t}+\varepsilon_{k}}(x^{k}) .
		\end{equation}
		Since $(x^{k})_{k \in \mathbb{N}} \in \overline{\Sigma_{\delta,\tilde{t}-\delta}^{g}}$, we have 
		\begin{equation}\label{bornexk}
			\delta < g((x^{k})')+ \delta \leq x_{N}^{k}\leq \tilde{t}-\delta,  \quad \forall k\geq 1,
		\end{equation}
		so the sequence $(x^{k}_{N})$ is bounded. Thus there exists $x^{\infty}_{N} \in [\delta,\tilde{t}-\delta]$ such that, up to a subsequence, $ x^{k}_{N} \to x^{\infty}_{N} $ as $k \to +\infty$.\\		
		Now we set
		\begin{equation}\label{defgk} 
				g_{k}(x')=g((x^{k})'+x'), \quad \forall x' \in \rnu \quad \forall k\geq 1.
		\end{equation}
Thanks to \eqref{bornexk} we have $0 \leq g_{k}(0')=g((x^{k})') \leq \tilde{t}$, for any $k \geq1$. 
The latter and the assumption $g \in \mathbb{C}^{1,\gamma}(\R^{N-1})$ immediately imply that the sequence $\left( g_{k}\right) $ is bounded in $C^{1,\gamma}(\mathcal{K})$, for any compact set $\mathcal{K} \subset \R^{N-1}$. 
Hence, there exists $g_{\infty} \in C^{1,\gamma}_{loc}(\R^{N-1})$ 
such that, up to a subsequence, $g_{k} \longrightarrow g_{\infty}$ in $C^1_{\text{loc}}(\R^{N-1})$.  

We also observe that, passing to the limit in \eqref{bornexk}, we obtain
\begin{equation}\label{controllo-AAinfinito}
\delta \leq g_{\infty}(0')+ \delta \leq x^{\infty}_N \leq \tilde{t}-\delta.
\end{equation}

\medskip

For any $k \geq 1,$ let us consider $\Omega^{k}=\{(x',x_N) \in \R^N, \, x_N>g_{k}(x') \}$,  the epigraph of $g_k$, $\Omega^{\infty}=\lbrace (x',x_N)\in \R^N, \, x_N >g_{\infty}(x') \rbrace$, the epigraph of $g_{\infty}$ and define 
	\begin{equation}\label{ext-u}
				\tilde{u}(x)=\left\{
				\begin{array}{crl}
				u(x) & \text{if}& x \in \overline{\Omega},\\
				0 & \text{if} & x \in \R^{N} \setminus \overline{\Omega}.
				\end{array}
				\right.
	\end{equation}
Clearly, $\tilde{u}$ is bounded and uniformly continuous on any finite strip $\{x_{N}<R\}$ so, the sequence $(\tilde{u}_{k})$ defined by
\begin{equation}\label{trasl-tilde-u_k}
\begin{array}{ccc}
\tilde{u}_{k}(x)=\tilde{u}(x'+(x^{k})',x_N) \qquad \forall \, x=(x',x_N)  \in \R^{N}, \quad \forall k \geq 1,
\end{array} 
\end{equation}
is relatively compact in $C^{0}_{\text{loc}}(\R^N)$ (by  Ascoli-Arzelà theorem).  \\
Therefore, there exists 
$\tilde{u}_{\infty} \in C^0(\R^N)$
such that, up to a subsequence, $\tilde{u}_{k} \to \tilde{u}_{\infty}$ in $C^{0}_{\text{loc}}(\R^N)$ and  
			\begin{equation}\label{eq18}
				\left\{
				\begin{array}{ccl}
					- \Delta \tilde{u}_{\infty}=f(\tilde{u}_{\infty}) & \text{in} & \mathcal{D}'(\Omega^{\infty}),\\ 
					\tilde{u}_{\infty}\geq 0 & \text{in} &  \Omega^{\infty},\\ 
					\tilde{u}_{\infty}=0 & \text{on} & \partial \Omega^{\infty},
				\end{array}
				\right.
			\end{equation} 
where the boundary condition follows by observing that, 
\begin{equation*}
\forall \, k \geq 1, \quad \forall \, x' \in \R^{N-1} \qquad 0 = \tilde{u}_{k}(x',g_k(x')) , 
\end{equation*}
and thus
\begin{equation*}
0 = \tilde{u}_{k}(x',g_k(x')) \longrightarrow \tilde{u}_{\infty}(x',g_{\infty}(x')) \qquad  \text{as }k
\rightarrow \infty,
\end{equation*}
thanks to the uniform convergence of $(\tilde{u}_{k})$ and  $(g_k)$ on compact sets.

By construction, $u_{\infty}:=\tilde{u}_{\infty_{\vert \Omega^{\infty}}} \in C^{0}(\overline{\Omega^{\infty}})$. Furthermore, $u_{\infty}  \in  C^2(\Omega^{\infty})$ by \eqref{eq18} and standard interior regularity theory for elliptic equations, and it satisfies
\begin{equation}\label{eq-u-infinito}
				\left\{
				\begin{array}{cll}
					-\Delta u_{\infty}=f(u_{\infty}) & \text{in} & \Omega^{\infty},\\ 
					u_{\infty} \geq 0 & \text{in} & \Omega^{\infty},\\ 
					u_{\infty}=0 & \text{on} & \partial \Omega^{\infty}.
				\end{array}
				\right.
\end{equation}

Now we deal with the Neumann boundary condition.

Since $g_{\infty}(0')<\tilde{t} - 2 \delta,$ by the continuity of $g_{\infty}$ and $\eqref{bornexk},$ we can choose $R>0$ small enough such that
\begin{equation}\label{eq22sup}
	\disp\sup_{x' \in \overline{B'(0',4R)}}g_{\infty}(x')< \tilde{t} - \delta,
\end{equation}
where $B'(0',4R) \subset \R^{N-1}.$ 
Since $g_{k} \to g_{\infty}$ in $C^{0}_{\text{loc}}(\R^{N-1}) $, there exists $k_{1}=k_{1}(R) \geq 1$ such that 
\begin{equation}\label{eq22bis}
	\forall k \geq k_{1}, \quad \forall x' \in \overline{B'(0',4R)}, \qquad g_{\infty}(x')-\dfrac{\delta}{4} \leq g_{k}(x') \leq  g_{\infty}(x')+\dfrac{\delta}{4}.
\end{equation}

Pick $T>2\tilde{t}+2R$ and for any $k \geq k_{1}$, we define on $ \overline{\Omega^{k}}$
\begin{equation*}
	u_{k}(x)=u(x'+(x^{k})',x_{N}).
\end{equation*}
Since $u$ solves \eqref{probleme}, then $u_{k} \in C^{2,\gamma}(\overline{\mathfrak{C}^{g_{k}}(0',2R,2T)})$ and 
\begin{equation}\label{equationestimé1-1}
	\left\{
	\begin{array}{ccc}
		-\Delta u_{k}=f(u_{k})  & \text{in} & \mathfrak{C}^{g_{k}}(0',2R,2T),\\
		u_{k}\geq 0 & \text{in} & \mathfrak{C}^{g_{k}}(0',2R,2T),\\
		u_{k}=0 & \text{on} & \dr \Omega^{k} \cap \widehat{\mathfrak{C}^{g_{k}}}(0',2R,2T),\\
		\dfrac{\partial u_k}{\partial \eta_k} = \mathfrak{c} & \text{on} & \dr \Omega^{k} \cap \widehat{\mathfrak{C}^{g_{k}}}(0',2R,2T),
	\end{array}
	\right.
\end{equation} 
where $\mathfrak{C}^{g_{k}}(0',r,\tau)= \left\lbrace x=(x',x_N)\in \R^N, \, x' \in B'(0',r) \, \text{ and } \,\, g_{k}(x')< x_N<\tau 
\right\rbrace$ and $\eta_k$ denotes the outward unit normal vector to $\partial \Omega^k$.

\noindent At this point, we apply Proposition 3.1. of \cite{bfs} with $u=u_k$, $\widetilde{R}=2R$, $\widetilde{\tau}=2T$, $r=\frac{3R}{4}$ and $t=\frac{5T}{4}$, to get the existence of constants $ \alpha_0=\alpha_0(N,g) \in (0,1)$ and $C_0=
C_0(R,T,N, g)>0$ such that, for any $ k \geq k_{1} $, 
\begin{equation}\label{est-ordre-zero}
\begin{split}
\|u_k \|_{C^{0,\alpha_0}(\overline{\mathfrak{C}^{g_{k}}(0',3R/4,5T/4)})} \leq C_0 (\|f'\|_{L^{\infty}([0,M_k])} + 
\vert f(0) \vert +1) (\| u_k\|_{L^{\infty}(\Omega\cap \{x_{N}<2T\})} +1),  
\end{split}	
\end{equation}
where $M_{k}=\displaystyle\sup_{\mathfrak{C}^{g_{k}}(0',2R,2T)}u_{k}$. \\
Since $M_{k} \leq \|u\|_{L^{\infty}(\Omega\cap \{x_{N}<2T\})}$, from the latter we infer that, for any $ k \geq k_{1} $,
\begin{equation}\label{est-ordre-zero-bis}
\begin{split}
\|u_k \|_{C^{0,\alpha_0}(\overline{\mathfrak{C}^{g_{k}}(0',3R/4,5T/4)})} \leq C_0 (\|f'\|_{L^{\infty}([0,M])} + 
\vert f(0) \vert +1) (M+1),  
\end{split}	
\end{equation}
where $M := \|u\|_{L^{\infty}(\Omega\cap \{x_{N}<2T\})} $.

Now we apply Theorem \ref{estimé_c_1} with $u=u_k$, $\widetilde{R}=2R$, $\widetilde{\tau}=2T$, $r=\frac{3R}{4}$ and $t=\frac{5T}{4}$, to deduce the existence of constants $ \alpha_1=\alpha_1(N,g) \in (0,1)$ and $C_1=C_1(R,T,N, \gamma)>0$ such that, for any $ k \geq k_{1} $, 
\begin{equation}\label{est-ordre-un}
\begin{split}
& \|\nabla u_{k}\|_{C^{0,\alpha_1}(\overline{\mathfrak{C}^{g_{k}}(0',3R/4,5T/4)})} \leq \|\nabla u_k \|_{L^{\infty}(\Omega\cap \{x_{N}<2T\})} + \left[ \nabla u_k \right]_{C^{0,\alpha_1}  (\overline{\mathfrak{C}^{g_{k}}(0',3R/4,5T/4)})} \\
& \leq \|\nabla u \|_{L^{\infty}(\Omega\cap \{x_{N}<2T\})} \\
& + C_1 \left[ \vert \mathfrak{c} \vert \left( 1+ \|\nabla g_k\|_{C^{0,\gamma}(\R^{N-1} ) } \right)^2  + 1 + \|f'\|_{L^{\infty}([0,M_k])} \right](1 + \|\nabla u_k\|_{L^{\infty}(\mathfrak{C}^{g_{k}}(0',2R,2T))}) \\
& \leq \|\nabla u \|_{L^{\infty}(\Omega\cap \{x_{N}<2T\})} \\
& + C_1 \left[ \vert \mathfrak{c} \vert \left( 1+ \|\nabla g\|_{C^{0,\gamma}(\R^{N-1} ) } \right)^2  + 1 + \|f'\|_{L^{\infty}([0,M])} \right](1 + \|\nabla u \|_{L^{\infty}(\Omega\cap \{x_{N}<2T\})} ) \\
\end{split}	
\end{equation}

Hence, by combining \eqref{est-ordre-zero-bis} and \eqref{est-ordre-un} we can find a constant $C'$, independent of $k$, such that 
\begin{equation}\label{Ascoli-Arzelà-grad}
\begin{split}
& \|u_{k}\|_{C^{1,\alpha}(\overline{\mathfrak{C}^{g_{k}}(0',3R/4,5T/4)})} \leq C' \qquad \forall \, k \geq k_{1},
\end{split}	
\end{equation}
 where $\alpha = \min \left\lbrace \alpha_0, \alpha_1, \gamma \right\rbrace \in (0,1)$. 
\medskip 
 
Set $\mathcal{K}:=\overline{B'(0',R/2) \times (-T,T)}$. We can then apply item (ii) of Lemma \ref{extension_C_1-cylinder} 
(with $r =3R/4$, $\widetilde{T}=5T/4$ and $t=T$) 
to see that any $u_{k}$ has an extension $v_{k} \in C^{1,\alpha}(\mathcal{K})$ satisfying 
  \begin{equation}\label{norm_c_1_alpha_extension}
  	\|v_{k}\|_{C^{1,\alpha}(\mathcal{K})}\leq C_{ext} (1+\|\nabla g_k \|_{C^{0,\alpha}(\R^{N-1})})^{4}\|u_{k}\|_{C^{1,\alpha}(\overline{\mathfrak{C}^{g_{k}}(0',3R/4,5T/4)})} \qquad \forall \, k \geq k_{1}, 
  \end{equation}
where $C_{ext}$ is a positive constant depending only on $N$ and $T$. Hence 
\begin{equation}\label{norm_c_1_alpha_extension-bis}
  	\|v_{k}\|_{C^{1,\alpha}(\mathcal{K})} \leq C' C_{ext} (1+\|\nabla g \|_{C^{0,\alpha}(\R^{N-1})})^{4} := 
  	C'' \qquad \forall \, k \geq k_{1}, 
  \end{equation}
where $C''$ is a constant independent of $k$. 

Thus, by Ascoli-Arzelà Theorem, there exists $v_{\infty} \in C^{1, \alpha}(\mathcal{K})$ such that, up to a subsequence,
\begin{equation}\label{cv-v_k}
	v_k \longrightarrow  v_{\infty}\quad \text{in } C^{1}(\mathcal{K}).
\end{equation}

Now we are ready to prove that $\dfrac{\partial u_{\infty}}{\partial \eta_{\infty}} = \mathfrak{c} $ on $\mathcal{K} \cap \partial \Omega^{\infty}$, where $\eta_{\infty}$ denotes the outward unit normal vector to $\partial \Omega^{\infty}$ (recall that 
$g_{\infty} \in C^1(\R^{N-1})$ by construction). 

Observe that, for any compact set $K\subset \R^{N-1}$ and any $x' \in K$ , we have
\begin{equation}\label{cv-normales}
\begin{split}
	|\eta_{k}(x',g_{k}(x'))-\eta_{\infty}(x',g_{\infty}(x'))|&=\Big|\frac{(\nabla g_{k}(x'),-1)^{T}}{\sqrt{1+|\nabla g_{k}(x')|^{2}}}- \frac{(\nabla g_{\infty}(x'),-1)^{T}}{\sqrt{1+|\nabla g_{\infty}(x')|^{2}}} \Big|\\
	&\leq  2 \vert \nabla g_{k}(x') -\nabla g_{\infty}(x')\vert \longrightarrow 0, \quad \text{as} \quad  k \longrightarrow \infty
\end{split}	
\end{equation}
since $ g_k \longrightarrow g_{\infty}$ in $ C^1_{loc}(\R^{N-1})$. 

Hence, for any $y \in \mathcal{K} \cap \partial \Omega^{\infty}$ we get
\begin{equation}\label{cv-finale1}
 \nabla v_{\infty}(y) \cdot \eta_{\infty}(y',g_{\infty}(y')) = \lim_{k \to \infty} \nabla v_k(y) \cdot \eta_k(y',g_k(y')).
\end{equation}

Also, for any $ k \geq k_{1}$ and any $y \in \mathcal{K} \cap \partial \Omega^{\infty}$ we have 
\begin{equation}\label{cv-finale2}
\begin{split}
& \nabla v_k(y) \cdot \eta_k(y',g_k(y')) = \nabla v_k(y) \cdot \eta_k(y',g_k(y')) - \nabla v_k(y', g_k(y')) \cdot \eta_k(y',g_k(y')) \\
& +\nabla v_k(y', g_k(y')) \cdot \eta_k(y',g_k(y')) = \left(  \nabla v_k(y) \cdot \eta_k(y',g_k(y')) - \nabla v_k(y', g_k(y')) \cdot \eta_k(y',g_k(y'))   \right) \\
& + \dfrac{\partial v_k}{\partial \eta_k} (y',g_k(y')) =\left(  \nabla v_k(y) \cdot \eta_k(y',g_k(y')) - \nabla v_k(y', g_k(y')) \cdot \eta_k(y',g_k(y'))  \right)  \\
& + \dfrac{\partial u_k}{\partial \eta_k} (y',g_k(y')) = \left(  \nabla v_k(y) \cdot \eta_k(y',g_k(y')) - \nabla v_k(y', g_k(y')) \cdot \eta_k(y',g_k(y'))  \right) + \mathfrak{c}, 
\end{split}	
\end{equation}
where in the latter line we have used that ${v_k}_{\vert \mathcal{K} \cap \overline{\Omega^{k}}} = u_k$ and \eqref{equationestimé1-1}. 

To proceed further, we note that 
\begin{equation}\label{cv-finale3}
\begin{split}
& \vert \nabla v_k(y) \cdot \eta_k(y',g_k(y')) - \nabla v_k(y', g_k(y')) \cdot \eta_k(y',g_k(y'))  \vert \leq 
\vert \nabla v_k(y)- \nabla v_k(y', g_k(y')) \vert  \\
& \leq C'' \vert (y',g_{\infty}(y')) - (y', g_k(y')) \vert^{\alpha} = C'' \vert  g_{\infty}(y') - g_k(y')) \vert^{\alpha} \longrightarrow 0, \quad \text{as} \quad k \longrightarrow \infty,
\end{split}	
\end{equation}
where in the latter line we have used \eqref{norm_c_1_alpha_extension-bis} and that $ g_k \longrightarrow g_{\infty}$ in $ C^1_{loc}(\R^{N-1})$. 

By putting together \eqref{cv-finale1}-\eqref{cv-finale3} and recalling that ${v_{\infty}}_{\vert \mathcal{K} \cap \overline{\Omega^{\infty}}} = u_{\infty}$  we obtain
\begin{equation}\label{cv-finale4}
\begin{split}
& \dfrac{\partial u_{\infty}}{\partial \eta_{\infty}} (y) = \dfrac{\partial v_{\infty}}{\partial \eta_{\infty}} (y) = \mathfrak{c}, \qquad\forall \, y \in \mathcal{K} \cap \partial \Omega^{\infty}.
\end{split}	
\end{equation}
By summarizing we have that $u_{\infty} \in C^{0}(\overline{\Omega^{\infty}}) \cap C^2(\Omega^{\infty}) \cap 
C^{1, \alpha}(\mathcal{K} \cap \overline{\Omega^{\infty}})$ is a classical solution to 
\begin{equation}\label{equation_v_infini}
	\left\{
	\begin{array}{cll}
		- \Delta u_{\infty}=f(u_{\infty}) & \text{in}& \Omega^{\infty}, \\
		u_{\infty} \geq 0 & \text{in} & \Omega^{\infty}, \\
		u_{\infty} = 0  & \text{on} & \partial \Omega^{\infty}, \\
		\dfrac{\partial u_{\infty}}{\partial \eta_{\infty}} =c & \text{on} & \mathcal{K} \cap \partial \Omega^{\infty}. 
	\end{array}
	\right.
\end{equation}

Now, by definition of $\tilde{t}$, we have
\begin{equation}\label{inégalité_extension_uk}
	u_{k}(x)\leq u_{k,\tilde{t}}(x) \qquad \forall x \in \overline{\mathfrak{C}^{g_{k}}(0',R/4,\tilde{t})} \subset \mathcal{K} \cap \overline{\Omega^{\infty}}, \quad \forall \, k \geq k_{1}.
\end{equation}  

By letting $k \longrightarrow + \infty$ in \eqref{inégalité_extension_uk} we deduce that 
\begin{equation}\label{inégalitéelimite}
	u_{\infty}(x) \leq u_{\infty,\tilde{t}}(x) \qquad \forall x \in \overline{\mathfrak{C}^{g_{\infty}}(0',R/4,\tilde{t})} .
\end{equation}

We notice that \eqref{hyp_contradiction} implies $u_{k}(0',x^{k}_{N}) > u_{k,\tilde{t}+\varepsilon_{k}}(0',x^{k}_{N}),$ for any 
$ k \geq 1$. Then, passing to the limit in the latter we deduce 
\begin{equation*}
	u_{\infty}(0',x^{\infty}_N)\geq u_{\infty,\tilde{t}}(0',x^{\infty}_N), 
\end{equation*}
and so 
\begin{equation*}
	u_{\infty}(0',x^{\infty}_N)=u_{\infty,\tilde{t}}(0',x^{\infty}_N)
\end{equation*}
since $(0',x^{\infty}_N) \in \mathfrak{C}^{g_{\infty}}(0',R/4,\tilde{t})$ and 
\eqref{inégalitéelimite} holds true.\\ 

Now, for any $x \in \mathfrak{C}^{g_{\infty}}(0',R/4,\tilde{t})$ we set
\begin{equation*}
	w_{\infty}(x)=u_{\infty,\tilde{t}}(x)-u_{\infty}(x),
\end{equation*}
then $w_{\infty}$ satisfies 
\begin{equation}\label{inégalitéelimite-fine}
	\left\{
	\begin{array}{cll}
		- \Delta w_{\infty}=f(u_{\infty,\tilde{t}})-f(u_{\infty}) \geq -L_{f}w_{\infty} & \text{in}& 
		\mathfrak{C}^{g_{\infty}}(0',R/4,\tilde{t}), \\
		w_{\infty} \geq 0  & \text{in} & \mathfrak{C}^{g_{\infty}}(0',R/4,\tilde{t}),\\
		(0',x^{\infty}_N) \in \mathfrak{C}^{g_{\infty}}(0',R/4,\tilde{t}) , & & w_{\infty}(0',x^{\infty}_N)=0,
	\end{array}
	\right.
\end{equation}
where $L_f$ is the Lipschitz constant of $f$ on the compact set $[0, \sup_{\, \overline{\mathfrak{C}^{g_{\infty}}(0',R/4,\tilde{t})}} u_{\infty,\tilde{t}}]$. \\
By the strong maximum principle we infer that $w_{\infty} \equiv 0$ on the set $\overline{\mathfrak{C}^{g_{\infty}}(0',R/4,\tilde{t})})$, which in turn gives 
\begin{equation}
	u_{\infty}(0',2\tilde{t}-g_{\infty}(0'))= u_{\infty,\tilde{t}}(0',g_{\infty}(0')) =u_{\infty}(0',g_{\infty}(0'))=0. 
\end{equation}
The latter shows that $u_{\infty}$ attains its minimum value in $\Omega^{\infty}$
at the point $(0',2\tilde{t}-g_{\infty}(0'))$, hence $\nabla u_{\infty}(0',2\tilde{t}-g_{\infty}(0')) =0$. 

Moreover, since $w_{\infty} \equiv 0$ on the set $\overline{\mathfrak{C}^{g_{\infty}}(0',R/4,\tilde{t})}$, we also have 
$$
\nabla u_{\infty}(0',2\tilde{t}-g_{\infty}(0'))=\nabla u_{\infty}(0',g_{\infty}(0')) = 0,
$$
and so 
\begin{equation*}
	0= \vert \nabla u_{\infty}(0',2\tilde{t}-g_{\infty}(0')) \vert = \vert \nabla u_{\infty}(0',g_{\infty}(0')) \vert = \vert \mathfrak{c} \vert \neq 0,
\end{equation*}
a contradiction. This proves Proposition \ref{propositionimportante}. \\
With Proposition \ref{propositionimportante} at our disposal, the conclusion of the proof of \textit{Step 2} is the same as that of the proof of Theorem $1.1$ of \cite{bfs}. 

Finally, the monotonicity of $u$ is obtained (via the Hopf's Lemma) exactly as in \textit{Step 3} of the proof of Theorem $1.1$ of \cite{bfs}.  \end{proof}

\medskip

\subsection{\textit{Rigidity and one-dimensional symmetry results}} \quad \\

With the monotonicity result established in Theorem \ref{th_overdet-f(0)<0}, it is immediate to extend Theorem \ref{Teo2-lip}  
 and Theorem \ref{Teo1-C1}  to the case $f(0) <0$. \\
The results presented below are completely new in the three-dimensional case. When $N=2$, they are new for unbounded solutions, since the case of bounded solutions is covered by the results proved in \cite{RRS17} (as already mentioned in the Introduction).
 
Specifically we have 

\smallskip

\begin{thm}\label{Teo2-lip-f(0)<0} Let $\Omega \subset \R^2$ be a globally Lipschitz-continuous epigraph bounded from below and defined by a function $g \in C^{3}(\R) \cap \mathbb{C}^{1,\gamma}(\R)$.  \\ 
Let $u \in C^1(\overline{\Omega}) \cap C^2(\Omega)$ be a  solution to \eqref{probleme}  with $\mathfrak{c} \neq 0$ and assume that one of the following hypotheses holds true :
\begin{align*}
	& \tag{H1} f \in {Lip}_{loc}([0,+\infty)), f(0)< 0 \text{ and }  \nabla u \in L^{\infty}(\Omega);
	\\[3pt]
	&\tag{H2}\left\{
	\begin{array}{ccc}
		f \in {Lip}_{loc}([0,+\infty)), \quad f(0)< 0, \quad \exists \, \zeta > 0 \, : \, f(t) \leq 0  \quad \textit{in}  \quad [\zeta,+\infty),\\
		\\
		\exists \, \alpha \in [0,1) \, : \quad  u(x)= O(\vert x \vert^{\alpha}) \qquad  \textit{as}  \quad \vert x \vert \longrightarrow \infty; \hfill
	\end{array}
	\right.
\end{align*}

Then, $\Omega=\R^2_+$ up to a vertical translation and there exists $u_{0}:[0,+\infty) \to (0,+\infty)$ strictly increasing such that
	\begin{equation*}
		u(x)=u_{0}(x_2) \quad \forall x \in \R^2_{+}.
	\end{equation*}
\end{thm}

\noindent and

\smallskip

\begin{thm}\label{Teo1-C1-f(0)<0} Let $\Omega \subset \R^3$ be a globally Lipschitz-continuous epigraph bounded from below and defined by a function $g \in  C^{3}(\R^2) \cap \mathbb{C}^{1,\gamma}(\R^2)$.  \\ 
Let $u \in C^1(\overline{\Omega}) \cap C^2(\Omega)$ be a  bounded solution to \eqref{probleme}  with $\mathfrak{c} \neq 0$. \\ 
Let $f \in C^1([0,+\infty))$ be such that $f(0) < 0$ and 
\begin{equation}
		F( \sup u) = \sup_{t \in [0, \sup u]} F(t).
\end{equation}

\noindent  Then, $\Omega=\R^3_{+}$ up to a vertical translation and there exists $u_{0}:[0,+\infty) \to (0,+\infty)$ strictly increasing such that
	\begin{equation*}
		u(x)=u_{0}(x_{3}) \quad \forall x \in \R^3_{+}.
	\end{equation*}
\end{thm}

\section{Some auxiliary results}\label{Appendix1}

This section is devoted to the demonstration of several results used to prove our main conclusions.  Some of them are general results of independent interest. They could be useful for further studies of the nonlinear Poisson equation on general domains. 

\subsection{Some uniform bounds}\label{Appendix-unif-bound}

\begin{prop}\label{prop_princ-maxBCN}
Assume $ N \geq 2$ and $ \alpha \in [0,1)$.  Let $\Omega \subset \R^{N}$ be an unbounded  domain contained in an open affine half-space.  \\
Let $f \in C^0(\R)$ be a function satisfying  
    \begin{equation}\label{ipo1}
\exists \, \zeta > 0 \quad : \quad f(t) \leq 0 \quad \textit{on } \quad [\zeta,+\infty).
	\end{equation}	
Let $u \in C^0(\overline{\Omega}) \cap C^2(\Omega)$ be a solution to
\begin{equation*}
		\left\{
		\begin{array}{ccc}
			-\Delta u \leq  f(u) & \text{in} & \Omega,\\
			 u \leq 0 & \text{on} & \dr \Omega,
		\end{array}
		\right.
	\end{equation*}
such that, $ u^+(x) = O (\vert x \vert^{\alpha})$ as $\vert x \vert \longrightarrow \infty$.  		

Then

\begin{itemize}

\item[1)] $ u \leq \zeta$ on $ \Omega$;

\medskip

\item[2)] $f(\sup_{\Omega} u) \geq 0$.  Moreover,  if $f(\sup_{\Omega} u) = 0$, then for any $ \mu < \sup_{\Omega} u$ 
there is $ t=t(\mu) \in (\mu,  \sup_{\Omega} u)$ such that $ f(t) >0$. 
\end{itemize}	
\end{prop}

\smallskip

\begin{rem}\label{rem-prop_princ-maxBCN} \quad \\
\noindent (i) The assumption $ \alpha <1$ is sharp in the above result. Indeed, $u(x) =x_N$ is a positive unbounded harmonic function on the half-space $\R^N_+$ with zero Dirichlet boundary condition. 

\noindent (ii) Proposition \ref{prop_princ-maxBCN} also shows that the only bounded solution $u \in C^0(\overline{\Omega}) \cap C^2(\Omega)$ to
\begin{equation*}
		\left\{
		\begin{array}{ccc}
			-\Delta u =0& \text{in} & \Omega,\\
			u \geq 0 & \text{in} & \Omega,\\
			 u =0 & \text{on} & \dr \Omega,
		\end{array}
		\right.
	\end{equation*}
must be $ u \equiv0$. 

Actually, the Remark following Lemma 2.1 in \cite{BCN3} implies that this conclusion holds true for any non dense open set of $ \R^2$. In particular, this result holds for any two-dimensional epigraph. 

\end{rem}    

\medskip

\begin{proof}  $1)$ We argue by contradiction.  Assume that there exists $x_0 \in \Omega$ such that $u(x_0)>\zeta$ and let 
$D \subset \Omega$ be the open connected component of the open set $\left\lbrace u >\zeta \right\rbrace$ which contains  the point $x_0$. Then, $u(x) > \zeta $ for any $x \in D$ and $u(x)=\zeta$ for any $x \in \partial D$ (since $ \partial D \subset \Omega$ thanks to the boundary condition satisfied by $u$ and the definition of $D$). 

Let $w=\zeta -u$. Then $w<0$ in $D$ and $w(x) = O (\vert x \vert^{\alpha})$, when $x \in D$ and $\vert x \vert \longrightarrow \infty$. Moreover, 
	\begin{equation*}
		\Delta w= -\Delta u \leq f(u) \leq 0 \quad \text{in}  \quad D.
	\end{equation*}
Then, the Remark\footnote{\, Indeed, since $\Omega$ is contained in an affine half-space, the exponent $ \alpha $, which appears in the barrier $g$ defined by equation (2.1) on p. 1094 of \cite{BCN3}, can be chosen arbitrarily in the open interval $(0,1)$. } following Lemma 2.1 in \cite{BCN3} yields $w \geq 0$ in $D$, a contradiction. Therefore, $ u \leq \zeta$ on $ \Omega$.

\medskip

If $f(\sup_{\Omega} u) < 0$, then there is $ \varepsilon >0$
such that $ f(t) < 0$ in $\left[ \sup_{\Omega} u -\varepsilon, \sup_{\Omega} u\right] $. Hence, $u$ is bounded above and satisfies
\begin{equation*}
		\left\{
		\begin{array}{ccc}
			-\Delta u \leq \tilde{f} (u) & \text{in} & \Omega,\\
			 u \leq 0 & \text{in} & \dr \Omega,
		\end{array}
		\right.
	\end{equation*}	
where $\tilde{f} \in C^0(\R)$ is defined by $\tilde{f}=f$ in $(-\infty, \sup_{\Omega} u ]$ and $\tilde{f} \equiv f(\sup_{\Omega} u) < 0 $ in $[ \sup_{\Omega} u, +  \infty)$. We can then apply item 1) to conclude that $ u \leq \sup_{\Omega} u -\varepsilon$ in $\Omega$.  A contradiction. The same argument proves the last claim. \end{proof}

\medskip

\begin{prop}\label{prop_K-O-B-type} 
Let $\Omega \subset \R^{N}$ be a proper domain.  Let $h \in C^0([0,+\infty))$ be a  function satisfying the Keller-Osserman conditions 
\begin{equation}\label{K-O-cond} 
           \left\{
           \begin{array}{ccc}
            h(0) =0, \quad h(t)>0 \quad \textit{for} \quad t>0, \quad \textit{h non-decreasing},\\
            & \\
            \int^{\infty} \Big( \int_0^t h(s) ds  \Big)^{-\frac{1}{2}} < \infty .
          \end{array}
           \right.
\end{equation}	
Let $v \in C^0(\overline{\Omega}) \cap H^1_{loc}(\overline{\Omega}) $ be a solution to
\begin{equation*}
		\left\{
		\begin{array}{ccc}
			\Delta v \geq h(v) & \text{in} & \mathcal{D'}(\Omega),\\
			 v \geq 0 & \text{on} &\Omega, \\
			 v=0 & \text{on} & \dr \Omega.
		\end{array}
		\right.
	\end{equation*}	
Then, $ v \equiv 0$ in $\Omega $. 	
\end{prop}

\medskip

\begin{proof}
We recall that, the Keller-Osserman condition \eqref{K-O-cond} implies that for any $ \varepsilon>0$ there exists 
$R_\varepsilon>0$ such that 
\begin{equation}\label{sol-expl}
\Delta w = h(w)  \quad \text{in}  \quad B(0,R_\varepsilon), 
\end{equation}
admits a $C^2$-radially symmetric, positive solution $w_\varepsilon$, which blows up at the boundary of the ball. Furthermore, 
$w_\varepsilon(0) = \varepsilon = \inf_{B(0,R_\varepsilon)}w_\varepsilon$.

For contradiction, suppose that $ v \not \equiv 0$ in $\Omega$. Then, there exists $ x_0 \in \Omega$ such that $v(x_0)>0$.
Let $ \varepsilon = \frac{v(x_0)}{2} >0$ and consider the corresponding solution $w_\varepsilon$ of \eqref{sol-expl}, satisfying the above properties.  Therefore,  the function $ u(x) = w_\varepsilon(\vert x- x_0 \vert)$ is a $C^2$-radially symmetric (with respect to $x_0$), positive solution to $\Delta w = h(w)$ in  $B(x_0,R_\varepsilon)$, which blows up at the boundary of the ball $B(x_0,R_\varepsilon)$. 

Let us consider the open set $U := B(x_0,R_\varepsilon) \cap \Omega$ and set $ M:= \sup_{\overline{U}} v$. Then, $M$ is finite, since $v \in C^0(\overline{\Omega})$ by assumption. Thus, there is $0 <R(M) < R_\varepsilon $ such that $u(x) > M$ on $ B(x_0,R_\varepsilon) \setminus B(x_0,R(M))$ (recall that $u$ blows up at the boundary of the ball $B(x_0,R_\varepsilon)$). The latter and the homogeneous Dirichlet condition satisfied by $v$,  imply that $v - u \leq 0 $ on 
$\partial \left[ B(x_0,R(M)) \cap \Omega\right] $. It follows that $(v-u)^+ \in H^1_0(B(x_0,R(M)) \cap \Omega)$ and so, we can take it as test function in the weak formulation of the equation satisfied by $v-u$ in $B(x_0,R(M)) \cap \Omega$, that is, in the weak formulation of $ - \Delta (v-u) \leq h(u) - h(v)$ in $B(x_0,R(M)) \cap \Omega$. This choice leads to 
$$
\int \vert \nabla (v-u)^+ \vert^2 = \int \nabla (v-u)  \nabla (v-u)^+ \leq  \int (h(u)-h(v))(v-u)^+ \leq 0
$$
since $h$ is non-decreasing. Consequently, $v \leq u$ on $B(x_0,R(M)) \cap \Omega$, by the continuity of $(v-u)^+$. In particular we have that $ 0< v(x_0) \leq u(x_0) = w_\varepsilon(0) = \varepsilon =  \frac{v(x_0)}{2}$. A contradiction. This proves the Proposition. 
\end{proof}

\medskip

\begin{cor}\label{cor_K-O-B-type} 
Let $\Omega \subset \R^{N}$ be a proper domain.  Let  $\tilde f \in C^0(\R)$ be a continuous function satisfying 
\begin{equation}\label{f-K-O-cond} 
           \left\{
           \begin{array}{ccc}
           \exists \, \mu \geq 0 \quad : \quad -\tilde f (t) \geq h(t) \quad \textit{for} \quad t>\mu, \\
           & \\
           \textit{where $h \in C^0([0,+\infty))$ satisfies the Keller-Osserman condition } \eqref{K-O-cond} .
          \end{array}
           \right.
\end{equation}	
Let $v \in C^0(\overline{\Omega}) \cap H^1_{loc}(\overline{\Omega}) $ be a solution to
\begin{equation}\label{subsol-f-K-0-cond}
		\left\{
		\begin{array}{ccc}
			- \Delta v \leq \tilde f (v) & \text{in} & \mathcal{D'}(\Omega),\\
			 v \leq \mu & \text{on} & \dr \Omega,
		\end{array}
		\right.
	\end{equation}	
then, $v \leq\mu$ in $\Omega $. 	

\noindent In particular, any $u \in C^1(\overline{\Omega}) \cap C^2(\Omega)$ solution to 
\begin{equation}\label{NonLin-PoissonEq-ancora}
\begin{cases}
-\Delta u=f(u) & \text{ in } \Omega,\\
\quad u \geq 0 & \text{ in } \Omega,\\
\quad u=0\,\, &\text{ on } \partial\Omega,
\end{cases}
\end{equation}
where $f$ is a specified Allen-Cahn type nonlinearity, satisfies $ u \leq 1$ in $ \Omega$. 
\end{cor}

\begin{proof} 
By assumption \eqref{f-K-O-cond} and Kato's inequality, we have 
$$
\Delta (v- \mu)^+ \geq - \tilde f(v) \textbf{1}_{\left\lbrace v > \mu \right\rbrace } \geq h(v) \textbf{1}_{\left\lbrace v > \mu \right\rbrace } \geq h((v- \mu)^+)  \quad \text{in} \quad \mathcal{D'}(\Omega)
$$ 
and $ (v- \mu)^+  = 0$ on $\partial \Omega$.  Then, $(v- \mu)^+  = 0$ on $\Omega$ by Proposition \ref{prop_K-O-B-type}.

The second part of the claim follows from the first one, by taking $\tilde f(t) = f(t+1)$, $h(t)=ct^p$, $\mu=0$ and $v= u-1$. Indeed, the assumption $\mathtt{W}4$) implies that \eqref{f-K-O-cond} is satisfied (recall that $p>1$ and $ c>0$ in $\mathtt{W}4$)).  \end{proof}

\begin{thm}\label{théorème_2.2}
	Let $\Omega \subset \R^{N}$ be a proper  open set of class $C^1$ and $\eta$ its outward unit normal. Assume $f \in L^{\infty}(\Omega)$ and let $v \in C^1(\overline{\Omega}) \cap C^2(\Omega)$ be a solution to 
	\begin{equation*}
		\left\{
		\begin{array}{ccc}
			-\Delta v=f  & \text{in} & \Omega,\\
			v \geq 0 & \text{in} & \Omega, \\
			v=0 & \text{on} & \dr \Omega.
		\end{array}
		\right.
	\end{equation*}
	Then, there exists a positive constant $C$, depending only on $N$, such that 
	\begin{equation*}
	\forall \, x \in \Omega \quad \qquad |\nabla v(x)|  \leq C \Big( \Big|\frac{\partial v}{\partial \eta}(x')\Big| +d(x,\partial \Omega) \|f\|_{L^{\infty}(B(x,d_x))}\Big),
	\end{equation*} 
	where $x' \in \partial \Omega$ is such that $|x-x'|=d(x,\partial \Omega)=: d_{x}$.
\end{thm}

\begin{proof}
	Let $x \in \Omega$ and $x' \in \partial \Omega$ such that $|x-x'|=d(x,\partial \Omega)$.\\
	According to Theorem 8.34 in \cite{gt}, there exists $u \in C^{1,\alpha}(\overline{B(x,d_{x})})$ a weak solution to
	\begin{equation*}
		\left\{
		\begin{array}{ccc}
			-\Delta u=f & \text{in} & B(x,d_{x}) ,\\
			u = 0 & \text{on} &  \partial B(x,d_{x}).
		\end{array}
		\right.
	\end{equation*}
	By Lemma \ref{lemme_gradient_partie_dirichlet}, there is $C_1=C_1(N)>0$ such that 
	\begin{equation}\label{dérivée_partie_dirichlet}
		\|\nabla u \|_{L^{\infty}(B(x,d_{x}))} \leq C_{1}d_{x} \|f\|_{L^{\infty}(B(x,d_{x}))}.
	\end{equation}
	Now we set $ h=v-u, $ then $h \in C^{2}(B(x,d_{x})) \cap C^{1}(\overline{B(x,d_{x})})$ and 
	\begin{equation*}
		\left\{
		\begin{array}{ccc}
			-\Delta h=0 & \text{in} & B(x,d_{x}) ,\\
			h \geq 0 & \text{on} &  \overline{B(x,d_{x})},
		\end{array}
		\right.
	\end{equation*} 
where the latter follows by the maximum principle, since $ h = v \geq 0 $ on $\partial B(x,d_{x}) \subset \overline{\Omega}$. 

	Therefore, by Lemma \ref{lemme_gradient_partie_harmonique}, there exists a positive constant $C_2= C_2(N)>0$ such that 
	\begin{equation}
		|\nabla h(x)| \leq C_{2} \Big|\frac{\partial h}{\partial \eta}(x')\Big|.\label{dérivée_h_harmonique}
	\end{equation}
	Finally, by \eqref{dérivée_partie_dirichlet} and \eqref{dérivée_h_harmonique}, we obtain
	\begin{equation*}
		\begin{split}
			|\nabla v(x)| &\leq |\nabla u(x)|+|\nabla h(x)| \leq  C_{1}d_{x} \|f\|_{L^{\infty}(B(x,d_{x}))} +C_{2} \Big|\frac{\partial h}{\partial \eta}(x')\Big|\\
			&  \leq C_{1} d_{x} \|f\|_{L^{\infty}(B(x,d_{x}))} +C_{2} \Big|\frac{\partial u}{\partial \eta}(x')\Big| + C_{2} \Big|\frac{\partial v}{\partial \eta}(x')\Big|\\
			& \leq (1+C_{2})C_{1}d_{x} \|f\|_{L^{\infty}(B(x,d_{x}))} + C_{2} \Big|\frac{\partial v}{\partial \eta}(x')\Big|\\
			& \leq \max\{C_{2},(1+C_{2})C_{1}\} \Big(d_{x}\|f\|_{L^{\infty}(B(x,d_{x}))}+\Big|\frac{\partial v}{\partial \eta}(x')\Big| \Big).
		\end{split}
	\end{equation*} 
	\end{proof}

\begin{cor}\label{cor_gradient}
Let $\Omega \subset \R^{N}$ be a proper  open set of class $C^1$. Let $f \in C^{0}(\R)$ and $u \in C^{2}(\Omega)\cap C^{1}(\overline{\Omega})$ be a bounded solution to \eqref{probleme}. \\
Then $\nabla u \in L^{\infty}(\Omega)$ and 
	\begin{equation*}
		\|\nabla u \|_{L^{\infty}(\Omega)}\leq C(|\mathfrak{c} |+\|u\|_{L^{\infty}(\Omega)} +\|f\|_{L^{\infty}([0,\|u\|_{L^{\infty}(\Omega)}])})
	\end{equation*}
where $C>0$ is a constant depending only on $N$.
\end{cor}
\begin{proof}
	Let $x  \in \Omega$. 
	
\smallskip

\noindent  $1)$ If $d(x,\partial \Omega) \geq 1$, then $B(x,1) \subset \Omega$. Hence, we can apply a standard interior gradient estimate (see for instance Theorem 3.9 in \cite{gt} ) to get 
	\begin{equation}\label{gradient_loin}
		|\nabla u(x)| \leq C_1(N) (\|u\|_{L^{\infty}(\Omega)}+ \|f\|_{L^{\infty}([0,\|u\|_{L^{\infty}(\Omega)}])}).
	\end{equation}

\noindent $2)$ If $d(x,\partial \Omega)<1$, let $x' \in \partial \Omega$ be such that $d(x,\partial \Omega)=|x-x'|$.
Then, Theorem \ref{théorème_2.2} implies that 
\begin{equation}\label{gradient_pret}
	|\nabla u(x)| \leq C_2(N)(|\mathfrak{c} |+ \|f\|_{L^{\infty}([0,\|u\|_{L^{\infty}(\Omega)}])}).
\end{equation} 
The desired conclusion then follows from \eqref{gradient_loin} and \eqref{gradient_pret}. \end{proof}

\subsection{Some uniform estimates}\label{Appendix-unif-estim}\quad \\

\noindent Assume $ N \geq2$. In the following, for a continuous function $h \,:\, \mathbb{R}^{N-1}\rightarrow \mathbb{R}$, we consider  its epigraph 
\[
\omega\,:=\, \left\lbrace x=(x',x_N) \in\mathbb{R}^{N-1}\times\mathbb{R} \,:\, x_N> h(x') \right\rbrace 
\]
and, for any $R>0$ and any $ \tau >\displaystyle\sup_{B'(0',R)} \, \vert h \vert$, we set 
		\begin{equation}\label{epi-cilindro}
			\mathfrak{C}^{h}(0',R,\tau)= \left\lbrace x=(x',x_N)\in \R^N, \, x' \in B'(0',R) \, \text{ and } \,\, h(x')< x_N<\tau 
			\right\rbrace , 
		\end{equation}
the intersection of the epigraph $\omega$ with the truncated cylinder $B'(0',R) \times \left( -\tau,\tau \right)$ and 
        \begin{equation*}
		\widehat{\mathfrak{C}^{h}}(0',R,\tau) = \mathfrak{C}^{h}(0',R,\tau) \cup 
         \left\lbrace x=(x',x_{N}) \in B'(0',R) \times \R, \, x_{N}=h(x') \right\rbrace \, .
		\end{equation*}

\bigskip

Then, the next uniform estime holds. 

\smallskip

\begin{thm}\label{estimé_c_1}
Assume $\sigma \in (0,1]$, $ c \in \R$, $\mathcal{E}>0$ and let $h \in \mathbb{C}^{1,\sigma}(\R^{N-1})$ be such that $\|\nabla h\|_{C^{0,\sigma}(\R^{N-1})} \leq \mathcal{E}$. \\
Let $\widetilde{R}>0$, $ \widetilde{\tau} > \max \left\lbrace \widetilde{R}, \, \displaystyle 4\sup_{\overline{B'(0',\widetilde{R})}} \vert h \vert\right\rbrace $ and 
$u \in C^1(\overline{\mathfrak{C}^{h}(0',\widetilde{R},\widetilde{\tau})}) \cap H^2({\mathfrak{C}^{h}(0',\widetilde{R},\widetilde{\tau})}) $ satisfy 
\begin{equation}\label{equationestimé1}
	\left\{
	\begin{array}{ccc}
		-\Delta u=f(u)  & \text{in} & \mathcal{D}'(\mathfrak{C}^{h}(0',\widetilde{R},\widetilde{\tau})),\\
		u=0 & \text{on} & \dr \omega \cap \widehat{\mathfrak{C}^{h}}(0',\widetilde{R},\widetilde{\tau}),\\
		\dfrac{\partial u}{\partial \eta}=c & \text{on} & \dr \omega \cap \widehat{\mathfrak{C}^{h}}(0',\widetilde{R},\widetilde{\tau}),
	\end{array}
	\right.
\end{equation}     
where $ f \in {Lip}_{loc}(\R)$. \\
Then, there exists $\alpha=\alpha(N,\mathcal{E}, \sigma) \in (0,1)$ such that $\nabla u \in 
C^{0,\alpha} (\overline{\mathfrak{C}^{h}(0',r,t)})$ for any 
$0<r<\frac{\widetilde{R}}{2}$ and any $\disp\sup_{B'(0',r)} |h| <t<\frac{3}{4}\widetilde{\tau}$,

\noindent and
\begin{equation}\label{estimeéelliptique1}
	\left[ \nabla u \right]_{C^{0,\alpha}  (\overline{\mathfrak{C}^{h}(0',r,t)})} \leq  C  
 \left[ |c| \left( 1+ \|\nabla h\|_{C^{0,\sigma}(\R^{N-1})} \right)^2  + 1 + \|f'\|_{L^{\infty}([-M,M])} \right]	
(1 + \|\nabla u\|_{L^{\infty}(\mathfrak{C}^{h}(0',\widetilde{R},\widetilde{\tau})))}) 
\end{equation}
where $C=C(r,\widetilde{R},\widetilde{\tau},\alpha,N) >0 $ and $M=\|u\|_{L^\infty (\mathfrak{C}^{h}(0',\widetilde{R},\widetilde{\tau}))}.$
\end{thm}

\medskip

The proof of the above theorem  is based on the following general result.

\medskip

\begin{prop}\label{prop0.1}
Let $\omega$ be the epigraph of a continuous function $h: \R^{N-1} \to \R$ and suppose that $\omega$ satisfies a uniform exterior cone condition (with reference cone $V$).\footnote{Let us recall that an open set $\omega \subset \R^N$ (not necessarily an epigraph) satisfies a \textit{uniform exterior cone condition} if for any $x_{0} \in \partial \omega$ there exists a finite right circular cone $V_{x_{0}},$ with vertex $x_{0},$ such that 
$\overline{\omega} \cap V_{x_{0}}=\{x_0 \}$ and the cones $V_{x_{0}}$ are all congruent to some fixed cone $V.$
The cone $V$ is called the \textit{reference cone}.} \\
Let $\widetilde{R}>0$, $ \widetilde{\tau} > \max \left\lbrace \widetilde{R} , \displaystyle\sup_{\overline{B'(0',\widetilde{R})}} \vert h \vert\right\rbrace $, $ f \in L^{N}(\mathfrak{C}^{h}(0',\widetilde{R},\widetilde{\tau}))$, 
$\Psi \in C^{0,\beta}(\dr \omega \cap \widehat{\mathfrak{C}^{h}}(0',\widetilde{R},\widetilde{\tau})) $ and 
	$v \in H^{1}(\mathfrak{C}^{h}(0',\widetilde{R},\widetilde{\tau})) \cap C^0(\overline{\mathfrak{C}^{h}(0',\widetilde{R},\widetilde{\tau})})$  satisfy
	\begin{equation}\label{equationgénérale}
		\left\{
		\begin{array}{ccc}
			-\Delta v=f  & \text{in} & \mathcal{D}'(\mathfrak{C}^{h}(0',\widetilde{R},\widetilde{\tau})),\\
			v=\Psi & \text{on} & \dr \omega \cap \widehat{\mathfrak{C}^{h}}(0',\widetilde{R},\widetilde{\tau}),\\
			
		\end{array}
		\right.
	\end{equation}     
	Then, there are constants $C=C(N,V,\widetilde{R},\beta)>0$ and $\alpha=\alpha(N,V,\beta) \in (0,1)$ such that for any $r>0$ and for any $x_0 \in \partial \omega \cap  \left( \overline{B'(0',\widetilde{R}/2)} \times \R \right) $,
	\begin{equation}\label{estimée_oscillation}
		\underset{\widehat{\mathfrak{C}^{h}}(0',\widetilde{R},\widetilde{\tau}) \cap B(x_{0},r)}{\text{osc}} v\leq C 
		( \left[ \Psi\right]_{ C^{0,\beta}(\dr \omega \cap \widehat{\mathfrak{C}^{h}}(0',\widetilde{R},\widetilde{\tau})) }+\|v\|_{L^{\infty}(\mathfrak{C}^{h}(0',\widetilde{R},\widetilde{\tau}))}+\| f\|_{L^{N}(\mathfrak{C}^{h}(0',\widetilde{R},\widetilde{\tau}))}) r^{\alpha}.
	\end{equation} 
\end{prop}

\medskip

\begin{proof}
Let $x,y \in \widehat{\mathfrak{C}^{h}}(0',\widetilde{R},\widetilde{\tau}) \cap B(x_{0},r)$.

\smallskip

If $r \geq \min \left\lbrace 1, \frac{\widetilde{R}}{2}\right\rbrace =: R_0>0$, then 
\begin{equation}\label{inegalite_r_grand}
	|v(x)-v(y)|\leq 2  \|v\|_{L^{\infty}(\mathfrak{C}^{h}(0',\widetilde{R},\widetilde{\tau}))} \leq 
	2 R_0^{-\gamma}  \|v\|_{L^{\infty}(\mathfrak{C}^{h}(0',\widetilde{R},\widetilde{\tau}))}r^{\gamma},
\end{equation}
for any $\gamma \in (0,1)$.\\

If $r < R_0,$ we apply Theorem 8.27 of \cite{gt} with $\Omega=\mathfrak{C}^{h}(0',\widetilde{R},\widetilde{\tau})$, 
$q=2N$ and $g=f$. Hence, there are positive constants $C'=C'(N,V)$ 
and $\theta=\theta(N,V) $ such that
\begin{equation*}
	\underset{\Omega \cap B(x_{0},r)}{\text{osc}} v\leq C' \left\lbrace r^{\theta}(R_{0}^{-\theta} \|v\|_{L^{\infty}(\Omega)}+ \|f\|_{L^{N}(\Omega)})+ \sigma(\sqrt{rR_{0}})\right\rbrace ,
\end{equation*}
where $\sigma(\sqrt{rR_{0}})=\underset{\partial \Omega \cap B(x_{0},\sqrt{rR_{0}})}{\text{osc}}v=\underset{\partial\omega \cap B(x_{0},\sqrt{rR_{0}})}{\text{osc}} \Psi \leq 2^{\beta} \left[ \Psi\right]_{C^{0,\beta}(\dr \omega \cap \widehat{\mathfrak{C}^{h}}(0',\widetilde{R},\widetilde{\tau})) }R_{0}^{\beta/2}r^{\beta/2}$.\\
Therefore, 
\begin{equation}\label{inegalite_r_petit}
	\begin{split}
& \underset{\Omega \cap B(x_{0},r)}{\text{osc}} v \leq C' (2^{\beta} + R_{0}^{-\theta}) (\left[ \Psi\right]_{C^{0,\beta}(\dr \omega \cap \widehat{\mathfrak{C}^{h}}(0',\widetilde{R},\widetilde{\tau})) }+ \|v\|_{L^{\infty}(\Omega)} +\|f\|_{L^{N}(\Omega)}) r^{\alpha},
	\end{split}
\end{equation}
where $ \alpha=\min(\theta, \frac{\beta}{2})$. 

The desired inequality \eqref{estimée_oscillation} then follows from \eqref{inegalite_r_grand} with $\gamma = \alpha$ and \eqref{inegalite_r_petit}. \end{proof}

\medskip

Now we are ready to prove Theorem \ref{estimé_c_1}.

\medskip

\noindent{\textit {Proof of Theorem \ref{estimé_c_1}.}}
	First, we observe that $w_{i}:=\frac{\partial u}{\partial x_i} \in H^{1}(\mathfrak{C}^{h}(0',\widetilde{R},\widetilde{\tau})) \cap C^0(\overline{\mathfrak{C}^{h}(0',\widetilde{R},\widetilde{\tau})})$ solves
	\begin{equation*}
			\left\{
		\begin{array}{ccc}
			-\Delta w_{i}=f'(u)w_{i}  & \text{in} & \mathcal{D}'(\mathfrak{C}^{h}(0',\widetilde{R},\widetilde{\tau})),\\
			w_{i}=c \eta_i & \text{on} & \dr \omega \cap \widehat{\mathfrak{C}^{h}}(0',\widetilde{R},\widetilde{\tau}),
		\end{array}
		\right.
	\end{equation*}
where $ \eta $ is the outward unit normal at $\partial \omega$. Since 
\begin{equation*}
	\Psi_i (x') := \eta_{i}(x', h(x'))=\left\{
	\begin{array}{ccc}
		 \frac{\partial_{i}h(x')}{\sqrt{1+|\nabla h(x')|^{2}}} & \text{if} & i \in \{1,\cdots, N-1\}, \\
		\frac{-1}{\sqrt{1+|\nabla h(x')|^{2}}} & \text{if} & i=N,
	\end{array}
	\right.
\end{equation*}
for any $x' \in \R^{N-1}$, it is easily seen that $\Psi_{i} \in C^{0,\sigma}(B'(0',\widetilde{R}))$ and

\begin{equation*}
	\left[ \Psi_{i}\right]_{C^{0,\sigma}(B'(0',\widetilde{R}))} \leq \left( 1+\|\nabla h\|_{L^{\infty}(\R^{N-1})} \right) 
	\left[ \nabla h \right]_{C^{0,\sigma}(\R^{N-1})} .
\end{equation*}
Hence,
\begin{equation}\label{semi-norme-w_i}
	\left[ {w_i}_{\vert \dr \omega \cap \widehat{\mathfrak{C}^{h}}(0',\widetilde{R},\widetilde{\tau})}\right] _{C^{0,\sigma}(
	\dr \omega \cap \widehat{\mathfrak{C}^{h}}(0',\widetilde{R},\widetilde{\tau}))} \leq |c| \left( 1+\|\nabla h\|_{C^{0,\sigma}(\R^{N-1})} \right)^2 .
\end{equation}

\smallskip 
 
Moreover, since $h$ is globally Lipschitz-continuous on $\R^{N-1}$, the epigraph $\omega$ satisfies a uniform exterior cone condition with reference cone $V$ depending only on $\|\nabla h \|_{L^{\infty}(\R^{N-1})}$, hence,  only on 
$\mathcal{E}$.\footnote{\, more precisely, we can choose the same reference cone $V=V(\mathcal{E})$, for any globally Lipschitz-continuous function  $h$ such that $\|\nabla h \|_{L^{\infty}(\R^{N-1})} \leq \mathcal{E}$.}  \\

Set $\Omega=\mathfrak{C}^{h}(0',\widetilde{R},\widetilde{\tau})$, $R_{1}=\min(\frac{\widetilde{R}}{2}-r,\frac{\widetilde{\tau}}{2})$, $M=\|u\|_{L^{\infty}(\Omega)}$ and pick $x,y  \in \mathfrak{C}^{h}(0',r,t)$ with $x\neq y $. \\

\textit{1.} If $|x-y|\geq \dfrac{R_{1}}{4}$ then, for any $\gamma \in (0,1)$, we have
\begin{equation}\label{estimée_1}
	\begin{split}
		|w_{i}(x)-w_{i}(y)| & \leq 2\|\nabla u\|_{L^{\infty}(\Omega)} =  2 \|\nabla u\|_{L^{\infty}(\Omega)} \Big(\frac{4}{R_{1}}\Big)^{\gamma}\Big(\frac{R_{1}}{4}\Big)^{\gamma} \\
		& \leq 2  \Big(\frac{4}{R_{1}}\Big)^{\gamma}  \|\nabla u\|_{L^{\infty}(\Omega)} |x-y|^{\gamma} : = M_1  
		\|\nabla u\|_{L^{\infty}(\Omega)} |x-y|^{\gamma}. 
	\end{split}
\end{equation}
 
\textit{2.} If $|x-y|< \dfrac{R_{1}}{4}$, then we consider the set  
$$T: = \left\lbrace x=(x',x_{N})\in \R^{N} \, : \, x'\in \overline{B'(0',\widetilde{R}/2)} \quad  \text{and} \quad  x_{N}=h(x') \right\rbrace $$
and, by observing that $x$ and $y$ play a symmetric role,  we distinguish (only) three cases. 

\medskip

\textit{Case 2.1} : $d(y,T)>\frac{R_{1}}{2}$.

\noindent In this case we have $\overline{B(y,\frac{R_{1}}{2})}\subset \mathfrak{C}^{h}(0',\widetilde{R},\widetilde{\tau}).$ To see this, we first observe that $\overline{B(y,\frac{R_{1}}{2})}\subset \omega$, thanks to the assumption $d(y,T)>\frac{R_{1}}{2}$ and by the definition of $R_1$, and thus, for any  $z \in \overline{B(y,\frac{R_{1}}{2})}$, we have 
\begin{equation*}
	h(z') < z_{N}=z_{N}-y_{N}+y_{N}< z_{N}-y_{N}+t \leq \frac{R_{1}}{2}+\frac{3\widetilde{\tau}}{4} < \frac{\widetilde{\tau}}{4}+ \frac{3\widetilde{\tau}}{4} = \widetilde{\tau},
\end{equation*}
and 
\begin{equation*}
	|z'|\leq |z'-y'|+|y'| \leq \frac{R_{1}}{2}+r< \frac{\widetilde{R}}{2}-r+r<\widetilde{R}. 
\end{equation*}
Now, by applying  Lemma \ref{Brandt-weak} with $\delta=\frac{R_{1}}{2}$ and $\gamma \in (0,1)$, we deduce
\begin{equation}\label{estimée_2}
	\begin{split}
		|w_{i}(x)-w_{i}(y)| &\leq \dfrac{\sqrt{N}}{2^{1-\gamma}}\Big(2N\|\nabla u\|_{L^{\infty}(\Omega)} 
		\Big(\frac{R_{1}}{2}\Big)^{-\gamma}+\|f'\|_{L^{\infty}([-M,M])}\|\nabla u\|_{L^{\infty}(\Omega)}\Big(\frac{R_{1}}{2}\Big)^{2-\gamma}\Big)
		|x-y|^{\gamma}\\
		&\leq M_{2}(1+\|f'\|_{L^{\infty}([-M,M])})\|\nabla u\|_{L^{\infty}(\Omega)}  |x-y|^{\gamma}, 
	\end{split}
\end{equation}
where $M_{2}=\dfrac{\sqrt{N}}{2^{1-\gamma}}\Big(2N \Big(\dfrac{R_{1}}{2}\Big)^{-\gamma}+\Big(\dfrac{R_{1}}{2}\Big)^{2-\gamma}\Big)$.\\

\textit{Case 2.2} : $d(x,T)\leq d(y,T) \leq \dfrac{R_{1}}{2}$ and $|x-y| \geq \frac{1}{4}d(y,T).$\\

Since $T$ is a compact set, there exists  $y_{0}=(y_{0}',h(y_{0}')) \in T$ such that $d(y,T)=|y-y_{0}|$. \\Therefore 
\begin{equation*}
	x,y \in B(y_{0},6|x-y|),
\end{equation*}
since 
\begin{align*}
	&|x-y_{0}|\leq |x-y|+|y-y_{0}|\leq  |x-y|+4|x-y|=5|x-y|,\\
	\intertext{and}
	&|y-y_{0}|=d(y,T)\leq 4|x-y|.
\end{align*}
Now, we apply Proposition \ref{prop0.1} with $r=6|x-y|$, $v=w_{i}$, $\Psi=c \eta_{i}$ and $x_{0}=y_{0}$, thus there exist constants $C=C(N,\mathcal{E},\widetilde{R},\sigma)>0$ and $\alpha=\alpha(N,\mathcal{E},\sigma) \in (0,1)$ such that  
\begin{equation}\label{estimé3}
	\begin{split}
		&|w_{i}(x)-w_{i}(y)| \leq \underset{\Omega \cap B(x_{0},r)}{\text{osc}} w_i \\
		& \leq 6^{\alpha} C \left( \left[{w_i}_{\vert \dr \omega \cap \widehat{\mathfrak{C}^{h}}(0',\widetilde{R},\widetilde{\tau})}\right] _{C^{0,\sigma}(\dr \omega \cap \widehat{\mathfrak{C}^{h}}(0',\widetilde{R},\widetilde{\tau}))}
 + \| w_{i} \|_{L^{\infty}(\Omega)} + \| f'(u)w_{i}\|_{L^{N}(\Omega)} \right) |x-y|^{\alpha} \\
         & \leq 6^{\alpha} C \left( |c| \left( 1+\|\nabla h\|_{C^{0,\sigma}(\R^{N-1})} \right)^2 + \|\nabla u\|_{L^{\infty}(\Omega)} + \| f'(u)w_{i}\|_{L^{N}(\Omega)} \right) |x-y|^{\alpha}, \\
	\end{split}
\end{equation}
where in the latter we have used \eqref{semi-norme-w_i}. 

Since 
\begin{equation*}
		\|  f'(u) w_{i} \|_{L^{N}(\Omega)}  \leq \|f'\|_{L^{\infty}([-M,M])}  \|\nabla u\|_{L^{\infty}(\Omega)}  \left( 2\widetilde{\tau}\mathcal{L}^{N-1}(B'(0',1)) \widetilde{R}^{N-1} \right)^{\frac{1}{N}}, 
\end{equation*}
where $\mathcal{L}^{N-1}(B'(0',1))$ denotes the measure of the unit ball of $\R^{N-1}$, we see that 
\begin{equation}\label{estimé3-bis}
		|w_{i}(x)-w_{i}(y)| \leq M_{3} M_3'
         (1 + \|\nabla u\|_{L^{\infty}(\Omega)})  |x-y|^{\alpha},
\end{equation}
where 
\begin{equation}\label{const-estimé3-bis}
		\left\{
		\begin{array}{lcc}
		M_{3}=6^{\alpha}C \left[1 + \left( 2\widetilde{\tau}\mathcal{L}^{N-1}(B'(0',1)) \widetilde{R}^{N-1} \right)^{\frac{1}{N}} \right],\\
		M_3' = \left[ |c| \left( 1+ \|\nabla h\|_{C^{0,\sigma}(\R^{N-1})} \right)^2  + 1 + \|f'\|_{L^{\infty}([-M,M])} \right] .
		\end{array}
		\right.
	\end{equation}     

\textit{Case 2.3}: $d(x,T)\leq d(y,T) \leq \dfrac{R_{1}}{2}$ and $ |x-y| < \frac{1}{4} d(y,T)$. \\ 

Let $y_{0} \in T$ such that $|y- y_{0}|=d(y,T)$. The function $v_{i}:=w_{i} - w_i(y_{0})$ satisfies
\begin{equation*}
	-\Delta v_{i} =f'(u)w_{i}.
\end{equation*}

By applying Lemma \ref{Brandt-weak} with $\delta=\frac{d(y,T)}{2}$ and $\gamma= \alpha$, we have 
\begin{equation}\label{brandt-split}
	\begin{split}
		|w_{i}(x)-w_{i}(y)|=|v_{i}(x)-v_{i}(y)| &\leq \frac{\sqrt{N}}{2^{1-\alpha}}\Big(\|f'\|_{L^{\infty}([-M,M])} \|\nabla u\|_{L^{\infty}(\Omega)} \Big(\frac{d(y,T)}{2}\Big)^{2-\alpha}\\
		&+2N \|v_{i}\|_{L^{\infty}(\overline{ B(y,\frac{d(y,T)}{2})})} \Big(\frac{d(y,T)}{2}\Big)^{-\alpha}\Big)  |x-y|^{\alpha} .
	\end{split}
\end{equation}
Now, we want to estimate  $\|v_{i}\|_{L^{\infty}(\overline{ B(y,\frac{d(y,T)}{2})})} $.  Let $ z \in \overline{ B(y,\frac{d(y,T)}{2})}$ and $y_{0}=(y'_{0},h(y'_{0}))\in T$ such that $d(y,T)=|y_{0}-y|.$ Then $z \in \Omega \cap B(y_{0},2d(y,T))$
since
\begin{equation*}
	|z-y_{0}|=|z-y+y-y_{0}|\leq |z-y|+|y-y_{0}| \leq \frac{3}{2}d(y,T) < 2d(y,T).
\end{equation*}
Thus we can apply Proposition \ref{prop0.1} with $r=2d(y,T)$, $v=v_{i}$, $\Psi=c(\eta_{i} - \eta_i(y_{0}))$ and $x_{0}=y_{0}$ to get
\begin{equation*}
	\begin{split}
& |v_{i}(z)|= |v_{i}(z) - v_{i}(y_0)| \\ 
& \leq 2^{\alpha} C  \left( \left[ {v_i}_{\vert \dr \omega \cap \widehat{\mathfrak{C}^{h}}
(0',\widetilde{R},\widetilde{\tau})} \right]_{C^{0,\sigma}(\dr \omega \cap \widehat{\mathfrak{C}^{h}}(0',\widetilde{R},\widetilde{\tau}))} + \| v_{i} \|_{L^{\infty}(\Omega)} +  \| f'(u)w_{i}\|_{L^{N}(\Omega)}\right)  
d(y,T)^{\alpha} \\
& \leq 2^{\alpha} C  \left( \left[ {w_i}_{\vert \dr \omega \cap \widehat{\mathfrak{C}^{h}}
(0',\widetilde{R},\widetilde{\tau})} \right]_{C^{0,\sigma}(\dr \omega \cap \widehat{\mathfrak{C}^{h}}(0',\widetilde{R},\widetilde{\tau}))} + 2 \| w_{i} \|_{L^{\infty}(\Omega)} +  \| f'(u)w_{i}\|_{L^{N}(\Omega)}\right)  
d(y,T)^{\alpha} \\
& \leq 2^{\alpha +1 } C\left( \left[ {w_i}_{\vert \dr \omega \cap \widehat{\mathfrak{C}^{h}}
(0',\widetilde{R},\widetilde{\tau})} \right]_{C^{0,\sigma}(\dr \omega \cap \widehat{\mathfrak{C}^{h}}(0',\widetilde{R},\widetilde{\tau}))} +  \| w_{i} \|_{L^{\infty}(\Omega)} +  \| f'(u)w_{i}\|_{L^{N}(\Omega)}\right)  
d(y,T)^{\alpha} \\
& \leq 3^{-\alpha} 2 M_3\left[ |c| \left( 1+ \|\nabla h\|_{C^{0,\sigma}(\R^{N-1})} \right)^2  + 1 + 
		\|f'\|_{L^{\infty}([-M,M])} \right]  (1 + \|\nabla u\|_{L^{\infty}(\Omega)})  d(y,T)^{\alpha} \\
& \leq 3^{-\alpha} 2 M_3 M_3'  (1 + \|\nabla u\|_{L^{\infty}(\Omega)})  d(y,T)^{\alpha} .		
	\end{split}
\end{equation*}
The latter and \eqref{brandt-split} yield
\begin{equation}\label{estimée_4}
	\begin{split}
		&|w_{i}(x)-w_{i}(y)| \leq \frac{\sqrt{N}}{2^{1-\alpha}}\Big(\|f'\|_{L^{\infty}([-M,M])} \|\nabla u\|_{L^{\infty}(\Omega)} \Big(\frac{d(y,T)}{2}\Big)^{2-\alpha} \\
		& + 4 N 3^{-\alpha}  M_3 M_3'  (1 + \|\nabla u\|_{L^{\infty}(\Omega)})  d(y,T)^{\alpha} 
		\Big(\frac{d(y,T)}{2}\Big)^{-\alpha}\Big)  |x-y|^{\alpha} \\
		& \leq \frac{\sqrt{N}}{2^{1-\alpha}}\Big(\|f'\|_{L^{\infty}([-M,M])} \|\nabla u\|_{L^{\infty}(\Omega)} \Big(\frac{R_1}{4}\Big)^{2-\alpha} + 4N M_3 M_3'  (1 + \|\nabla u\|_{L^{\infty}(\Omega)}) \Big)  |x-y|^{\alpha} \\
& \leq \frac{\sqrt{N}}{2^{1-\alpha}}\Big(M_3' \Big(\frac{R_1}{4}\Big)^{2-\alpha} + 4N M_3 M_3' \Big) 
(1 + \|\nabla u\|_{L^{\infty}(\Omega)}) |x-y|^{\alpha} \\
& \leq M_4 \left[ |c| \left( 1+ \|\nabla h\|_{C^{0,\sigma}(\R^{N-1})} \right)^2  + 1 + \|f'\|_{L^{\infty}([-M,M])} \right]	
(1 + \|\nabla u\|_{L^{\infty}(\Omega)}) |x-y|^{\alpha}, 	
	\end{split}
\end{equation}
where $M_4 = \frac{\sqrt{N}}{2^{1-\alpha}}\Big(\Big(\frac{R_1}{4}\Big)^{2-\alpha} + 4N M_3\Big)$. 

\noindent The desired conclusion \eqref{estimeéelliptique1} then follows from \eqref{estimée_1}, \eqref{estimée_2} with $ \gamma = \alpha$, and \eqref{estimé3-bis}, \eqref{estimée_4} by taking  $C=\max(M_{1},M_{2},M_{3},M_{4})$.
\qed

\subsection{Some energy estimates}\label{Appendix-energy-est} \quad \\

\noindent In this subsection we employ the strategy developed in \cite{FVarma} (see Section 9 therein).  To this end we  recall that have denoted by $F$ the primitive of $f$ vanishing at $0$.

\begin{lem}\label{lem-energy-est-N-d}
Assume $c \in \R$ and $N \geq 2$. Let $u \in C^1(\overline{\Omega}) \cap C^2(\Omega)$ be a  bounded solution to \eqref{probleme} where $\Omega\subset \R^N$ is an epigraph with boundary of class $C^1$ and $f \in {Lip}_{loc}([0,+\infty))$.  

\noindent If $u$ is monotone, i.e., 
\begin{equation}\label{hyp-monot-N}
		\frac{\partial u}{\partial x_N}(x)>0 \qquad \forall x \in \Omega,
\end{equation}
then there exists a constant $C>0$, depending only on $ \mathfrak{c}$, $N$, $f$ and $ \Vert u \Vert_{L^{\infty}(\Omega)}$, such that, for any $ R>0$,  
	\begin{equation}\label{stima-energia-N-dim}
     \int_{B(0,R)\cap \Omega}  \frac{|\nabla u|^{2}}{2} + (c- F(u))  \leq C \left(\mathcal{H}^{N-1} (\partial (\Omega \cap B(0,R))) + \int_{B(0,R)\cap \Omega} \frac{|\nabla \overline{u} |^{2}}{2} + (c- F(\overline{u} ))\right),
	\end{equation}  
where\footnote{\, $\overline{u}$ is well-defined on $\R^N$ since $u$ is bounded and monotone on $\Omega$. Also recall that, by standard elliptic estimates, $\overline{u}$ solves the equation $ - \Delta \overline{u} = f(\overline{u}) $  on $\R^N$, and thus also on $\R^{N-1}$.} $ \, \overline{u}(x',x_N) =  \overline{u}(x') = : \lim_{x_N \to +\infty} u(x',x_N)$.
\end{lem}

\begin{proof}  From Corollary \ref{cor_gradient} there exists a constant $M>0$, depending only on $ \mathfrak{c}$, $N$, $f$ and $ \Vert u \Vert_{L^{\infty}(\Omega)}$,  such that
	\begin{equation}\label{u_et_gradient_u}
		u(x)+|\nabla u(x)| \leq M \quad \text{for any} \quad x \in \overline{\Omega}.
	\end{equation}

As in \cite{FVarma}, for any $R>0$ we consider the energy
	\begin{equation*}
	\mathcal{E}_{R,c}(u):= \disp\int_{B(0,R) \cap \Omega} \frac{|\nabla u|^{2}}{2} + (c - F(u))
	\end{equation*}
and, for any $x=(x',x_N) \in \overline{\Omega}$ and $t \geq 0$, we define
	\begin{equation*}
		u^{t}(x',x_N):=u(x',x_N+t).
	\end{equation*}

We can therefore proceed as in Lemma 9.1. of \cite{FVarma} to get, for any $R,T>0$,  
\begin{equation*}
\mathcal{E}_{R,c}(u^T) - \mathcal{E}_{R,c}(u) \geq -2M^2 	 \int_{\partial (\Omega \cap B(0,R))} d \mathcal{H}^{N-1} = -2M^2 \mathcal{H}^{N-1} (\partial (\Omega \cap B(0,R))).
\end{equation*}
The desired conclusion then follows by letting $T \longrightarrow \infty$ in the latter inequality, for any fixed $R>0$. 
\end{proof}

\smallskip

When $N=3$ and $f$ is of class $C^1$, the above result can be better specified as follows : 

\smallskip

\begin{lem}\label{lem-energy-est-3-d}
Let $u \in C^1(\overline{\Omega}) \cap C^2(\Omega)$ be a  bounded solution to \eqref{probleme} where $\Omega\subset \R^3$ is an epigraph with boundary of class 
$C^1$. Let $f \in C^1([0,+\infty))$ be such that 
\begin{equation}\label{c(u)-bis}
		F( \sup u) = \sup_{t \in [0, \sup u]} F(t).
\end{equation}

\noindent If $u$ is monotone, i.e., 
\begin{equation}\label{hyp-monot-3}
		\frac{\partial u}{\partial x_3}(x)>0 \qquad \forall x \in \Omega,
\end{equation}
then there exists a constant $C>0$, depending only on $ \mathfrak{c}$, $f$ and $ \Vert u \Vert_{L^{\infty}(\Omega)}$, such that 
	\begin{equation}\label{stima-energia-3-dim}
		\disp\int_{B(0,R)\cap \Omega}|\nabla u|^{2} \leq C \left(  R^{2} +  \mathcal{H}^{2} (\partial (\Omega \cap B(0,R)) \right)  \qquad \forall  \,R>0.
	\end{equation}  
\end{lem}

\begin{proof} Set $c_u = \sup_{t \in [0, \sup u]} F(t)$. Then, from Lemma \ref{lem-energy-est-N-d}, applied with $c= c_u$ and $N=3$,  we have for any $R>0$,  
\begin{equation}\label{stima-energia-3-dim-part}
         \int_{B(0,R)\cap \Omega}  \frac{|\nabla u|^{2}}{2} + (c- F(u))  \leq C \left(\mathcal{H}^{2} (\partial (\Omega \cap B(0,R))) + \int_{B(0,R)\cap \Omega} \frac{|\nabla \overline{u} |^{2}}{2} + (c- F(\overline{u} ))\right).
	\end{equation}  
Now we estimate $\int_{B(0,R)\cap \Omega} \frac{|\nabla \overline{u} |^{2}}{2} + (c- F(\overline{u} )),$ for any $ R>0. $ To this end, we first observe that the monotonicity assumption \eqref{hyp-monot-3} implies that $u$ is a stable solution to $- \Delta u=f(u)$ in  $\Omega$. Hence, since $f$ is of class $C^1$, the limit profil $ \overline{u} $ is a classical \textit{stable solution} to
\begin{equation}\label{equationlimite}
		\left\{
		\begin{array}{ccc}
			-\Delta \overline{u} =f(\overline{u} )  & \text{in} & \R^{2},\\
			 0 < \overline{u} \leq  \sup u& \text{in} & \R^{2},\\
			 \sup_{\R^2} \overline{u}  = \sup_{\Omega} u .
		\end{array}
		\right.
	\end{equation}
It is well-known (cf. \cite{dan1}, see also \cite{fsv} for related results) that any bounded entire two-dimensional stable solution to $-\Delta v=f(v)$, with $f \in C^1$, is either constant or one-dimensional and strictly monotone with respect to some fixed direction. 

If $ \overline{u} $ is constant, we have $ \overline{u} \equiv \sup_{\R^2} \overline{u}  = \sup_{\Omega} u $ and so $F(\overline{u})  = F(\sup u ) = c_u$, by assumption. Hence, $\frac{|\nabla \overline{u} |^{2}}{2} + (c_u- F(\overline{u} )) \equiv 0$ and $\int_{B(0,R)\cap \Omega}\frac{|\nabla \overline{u} |^{2}}{2} + (c_u - F(\overline{u} )) =0$ for any $ R>0$.  
Then, \eqref{stima-energia-3-dim} follows immediately from \eqref{stima-energia-3-dim-part}. 

If $ \overline{u} $ is not constant, then after a rotation of the coordinates, we have 
	\begin{equation}\label{symdim1}
		\overline{u}(x_{1},x_{2})=\overline{u}(x_{1}) \quad \text{for any}\, (x_{1},x_{2})\in\R^{2}
	\end{equation}
and $\overline{u} $ is a strictly monotone solution to the ODE 
\begin{equation}\label{eqlimitedim1}
		-\overline{u}^{''}=f(\overline{u}) \quad \text{in} \quad \R. 
	\end{equation}
Therefore, 
\begin{equation}\label{egalité_dim_1}
	\frac{(\overline{u}^{'}(s))^{2}}{2}+F(\overline{u}(s))=F(\sup \overline{u}) \qquad \forall s \in \R, 
\end{equation}
and so, by \eqref{equationlimite}, we also have 
\begin{equation}\label{egalité_dim_1-bis}
\frac{(\overline{u}^{'}(s))^{2}}{2}+F(\overline{u}(s))=F(\sup \overline{u})=c_{u} \qquad \forall s \in \R.
\end{equation}
From the latter we infer that 
\begin{equation}\label{inegalité_dim_1-bis}
c_u- F(\overline{u}(s)) >0  \qquad \forall s \in \R,
\end{equation}
and 
\begin{equation}\label{inegalité_dim_1-tris}
\frac{(\overline{u}^{'}(s))^{2}}{2} = c_u- F(\overline{u}(s)) \leq F(\sup u) - \inf_{t \in [0, \sup u]} F(t) = C(f, \sup u) \qquad \forall s \in \R,
\end{equation}
hence, 
\begin{equation*}
\begin{split}
& \int_{B(0,R)\cap \Omega}\frac{|\nabla \overline{u} |^{2}}{2} + (c_u- F(\overline{u} )) 
\leq \int_{B(0,R)}\frac{|\nabla \overline{u} |^{2}}{2} + (c_u- F(\overline{u} )) \\ 
& \leq \disp\int_{(-R,R)^3} \frac{|\nabla \overline{u} |^{2}}{2} + (c_u- F(\overline{u} )) = R^2 \disp\int_{-R}^R \left[ \frac{ (\overline{u}'(x_{1}))^{2}}{2} + (c_u- F(\overline{u}(x_1))) \right] dx_{1} .
\end{split}
\end{equation*}	
From \eqref{egalité_dim_1-bis} we deduce that
\begin{equation}\label{inegalité_energie-dim_1-1}
\disp\int_{B(0,R)}\frac{|\nabla \overline{u} |^{2}}{2} + (c_u- F(\overline{u} )) \leq R^2 \disp\int_{-R}^R (\overline{u}'(x_{1}))^{2}  dx_{1} \leq 
R^2 \Vert \overline{u}' \Vert_{L^{\infty}(\R)} \disp\int_{-R}^R \vert \overline{u}'(x_{1}) \vert dx_{1} .
\end{equation}	
If $\overline{u}'>0$, then
\begin{equation}\label{inegalité_energie-dim_1-caso1}
\disp\int_{-R}^R \vert \overline{u}'(x_{1}) \vert dx_{1} = \disp\int_{-R}^R  \overline{u}'(x_{1}) dx_{1} = 
\overline{u}(R)- \overline{u}(-R) \leq 2  \Vert \overline{u} \Vert_{L^{\infty}(\R)} , 
\end{equation}	
if $\overline{u}'<0$, then
\begin{equation}\label{inegalité_energie-dim_1-caso2}
\disp\int_{-R}^R \vert \overline{u}'(x_{1}) \vert dx_{1} = \disp\int^{-R}_R  \overline{u}'(x_{1}) dx_{1} = 
\overline{u}(-R)- \overline{u}(R) \leq 2  \Vert \overline{u} \Vert_{L^{\infty}(\R)} .
\end{equation}	
In both cases we have 
\begin{equation}\label{inegalité_energie-finale}
\disp\int_{B(0,R)}\frac{|\nabla \overline{u} |^{2}}{2} + (c_u- F(\overline{u} )) \leq 2  \Vert \overline{u}' \Vert_{L^{\infty}(\R)}  \Vert \overline{u} \Vert_{L^{\infty}(\R)} R^2
\end{equation}	
and so, from \eqref{stima-energia-3-dim-part}, \eqref{c(u)-bis} and \eqref{inegalité_energie-finale}, we deduce that
\begin{equation*}
\begin{split}
    & \int_{B(0,R)\cap \Omega}\frac{|\nabla u |^{2}}{2} \leq  \int_{B(0,R)\cap \Omega}\frac{|\nabla u |^{2}}{2} + (c_u- F(u)) \\
    &\leq C \left(\mathcal{H}^{2} (\partial (\Omega \cap B(0,R))) + \int_{B(0,R)}\frac{|\nabla \overline{u} |^{2}}{2}+ (c_u- F(\overline{u} ))\right) \\
    & \leq C \left( \mathcal{H}^{2} (\partial (\Omega \cap B(0,R))) + 2  \Vert \overline{u}' \Vert_{L^{\infty}(\R)}  \Vert \overline{u} \Vert_{L^{\infty}(\R)} R^2 \right), 
\end{split}   
	\end{equation*}  
which, in view of  \eqref{inegalité_dim_1-tris}, concludes the proof. 
\end{proof}

\subsection{Appendix}\label{Appendix}

\begin{lem}\label{lemme_gradient_partie_dirichlet}
	Assume $ \alpha \in (0,1)$, $x_{0}\in \R^{N}$, $R>0$ and $f \in L^{\infty}(B(x_{0},R))$. \\
	Let $u \in C^{1,\alpha}(\overline{B(x_{0},R)})$ be a weak solution to  
	\begin{equation*}
		\left\{
		\begin{array}{ccc}
			-\Delta u=f & \text{in}& B(x_{0},R),\\
			u=0 & \text{on} & \dr B(x_{0},R).
		\end{array}
		\right.
	\end{equation*}
	Then, there exists a positive constant $C$, depending only on $N$,  such that 
	\begin{equation*}
		\|\nabla u\|_{L^{\infty}(B(x_{0},R))} \leq C  \|f\|_{L^{\infty}(B(x_{0},R))}R.
	\end{equation*}
\end{lem}
\begin{proof}
The function 
	\begin{equation*}
		w(x)=u(x_{0}+Rx), \qquad x \in B_{1}:=B(0,1),
	\end{equation*}
 is a weak solution to 
\begin{equation*}
		\left\{
		\begin{array}{ccc}
			-\Delta w= g & \text{in}& B_{1},\\
			w=0 & \text{on} & \dr B_{1},
		\end{array}
		\right.
	\end{equation*}	
where $g(x) = R^{2}f(x_{0}+Rx)$, $ x \in B_{1}$. 

Therefore, according to Theorem $8.33$ in \cite{gt}, there is a constant $\widetilde{C}>0$, depending only on $N$, such that
	\begin{equation}\label{gradient_fonction_renormalisée}
		\|\nabla w \|_{L^{\infty}(B_{1})}\leq \widetilde{C}(\|w\|_{L^{\infty}(B_{1})}+\|g\|_{L^{\infty}(B_{1})}). 
	\end{equation}
Now, the function
	\begin{equation*}
		v(x)=w(x)+\frac{\|g\|_{L^{\infty}(B_{1})}}{2N}|x|^{2}, \qquad x \in B_{1}
	\end{equation*}
satisfies $v \geq w$ on $ B_{1}$ and 
	\begin{equation*}
		-\Delta v =-\Delta w -\|g\|_{L^{\infty}(B_{1})}=g -\|g\|_{L^{\infty}(B_{1})} \leq 0. 
	\end{equation*}
Hence, the weak maximum principle (see for instance Theorem 8.1 in \cite{gt}) yields 
\begin{equation}\label{Principe_du_max_pour_w}
	w(x) \leq \disp\sup_{\partial B_{1}} v =\frac{\|g\|_{L^{\infty}(B_{1})}}{2N}, \qquad x \in B_{1}, 
\end{equation}
and so, 
\begin{equation}\label{gradient_fonction_renormalisée-bis}
		\|\nabla w \|_{L^{\infty}(B_{1})}\leq 2 \widetilde{C}\|g\|_{L^{\infty}(B_{1})} \leq 2 \widetilde{C}  R^{2} \|f\|_{L^{\infty}(B(x_{0},R))}.
	\end{equation}
Then, by definition of $w$ we get 
	\begin{equation*}
\|\nabla u\|_{L^{\infty}(B(x_{0},R))}=\frac{\|\nabla w\|_{L^{\infty}(B_{1})}}{R} \leq 2 \widetilde{C}\|f\|_{L^{\infty}(B(x_{0},R))}R.
	\end{equation*} \end{proof}

\bigskip

\begin{lem}\label{lemme_gradient_partie_harmonique}
	Let $x_{0}\in \R^{N}$, $R>0$ and $h \in C^{2}(B(x_{0},R)) \cap C^{1}(\overline{B(x_{0},R)})$ be a solution to
	\begin{equation}\label{h_harmonique}
		\left\{
		\begin{array}{ccc}
			-\Delta h=0 & \text{in} & B(x_{0},R) ,\\
			h \geq 0 & \text{in} &  \overline{B(x_{0},R)}.
		\end{array}
		\right.
	\end{equation}

	\noindent If $x \in \partial B(x_{0},R)$ and $h(x)=0$, then there is a positive constant $C$, depending only on $N$, such that 
	\begin{equation*}
		|\nabla h(x_{0})| \leq C \Big|\frac{\partial h}{\partial \eta}(x)\Big|,
	\end{equation*}
where $\eta$ is the outward unit normal at  $\partial B(x_{0},R)$.
\end{lem}
\begin{proof}
\noindent As in the previous Lemma we can suppose that $x_{0}=0$ and $R=1$. Also, we can suppose that $h(0)>0$, otherwise the result is trivially true by the strong maximum principle. \\
By Harnack inequality, there is a constant $C_{H}=C_H(N) >0$ such that
 	\begin{equation*}
 	h(0)\leq \disp\sup_{\overline{B(0,1/2)}} h\leq C_{H} \disp\inf_{\overline{B(0,1/2)}} h \leq  C_{H} \disp\inf_{\partial B(0,1/2)} h,
 	\end{equation*}
where the latter inequality follows from the harmonicity of $h$ and the maximum principle.\\
 	
 	\noindent Now, we define in $\overline{B(0,1)} \backslash \{0\}$, the function
 	 \begin{equation*}
 	 	v(x)=\left\{
 	 	\begin{array}{ccc}
 	 		- \ln(|z|)& \text{if} & N=2 ,\\
 	 		\frac{1}{|z|^{N-2}}-1 & \text{if} &  N \geq 3,
 	 	\end{array}
 	 	\right.
 	 \end{equation*}
and we recall that $v$ solves 
   	\begin{equation}\label{sol_fondamentale}
   	\left\{
   	\begin{array}{ccc}
   		-\Delta v=0 & \text{in} & B(0,1)\backslash \{0\} ,\\
   		v=0 & \text{in} &  \partial B(0,1).
   	\end{array}
   	\right.
   \end{equation}
Suppose $N \geq 3$. If $y \in \partial B(0,1/2)$, then
\begin{equation}\label{h_sur_le_bord_la_demi_boule}
	h(y) \geq \disp\inf_{\partial B(0,1/2)} h \geq \frac{h(0)}{C_{H}}  \frac{v(y)}{2^{N-2}-1}= \frac{h(0)}{C_{H}(2^{N-2}-1)}v(y). 
\end{equation}
We let $\varepsilon=\frac{h(0)}{C_{H}(2^{N-2}-1)}>0$, $A=B(0,1) \backslash \overline{B(0,1/2)}$ and we consider the function 
\begin{equation*}
w(z)=h(z)-\varepsilon v(z), \qquad \forall \, z \in A.
\end{equation*}
 Hence, \eqref{h_harmonique}, \eqref{sol_fondamentale} and \eqref{h_sur_le_bord_la_demi_boule} ensure that 
\begin{equation*}
	\left\{
	\begin{array}{ccc}
		-\Delta w=0 & \text{in} & A,\\
		w\geq 0 & \text{on} &  \partial A,
	\end{array}
	\right.
\end{equation*}
and the maximum principle yields 
\begin{equation*}
	w \geq 0 \quad \text{in} \quad A.
\end{equation*}
Since $w(x)=0$, we have 
\begin{equation}\label{dérivée_normale_de_w}
\dfrac{\partial h}{\partial \eta}(x) - \varepsilon \dfrac{\partial v}{\partial \eta}(x) = \dfrac{\partial w}{\partial \eta}(x) \leq 0. 
\end{equation}
However, here $\eta=x$, thus
\begin{equation}\label{dérivée_normale_de_v}
	\dfrac{\partial v}{\partial \eta}(x) =(\nabla v(x) \cdot x)=2-N.
\end{equation}
Finally, thanks to \eqref{dérivée_normale_de_w} and \eqref{dérivée_normale_de_v}, we obtain
\begin{equation}\label{comparaison_h(0)_et_dérivée_normale}
	\Big|\frac{\partial h}{\partial \eta}(x)\Big| \geq - \dfrac{\partial h}{\partial \eta}(x)  \geq \frac{N-2}{C_{H}(2^{N-2}-1)} h(0)>0.
\end{equation}
Now, for any $j=1,...,N$, by the mean value Theorem for harmonique functions, and the positivity of $h$,  we have
\begin{equation*}
\Big \vert \frac{\partial h }{\partial z_j} (0) \Big \vert = \Big \vert \frac{1}{\omega_N} \int_{B(0,1)}  \frac{\partial h }{\partial z_j} dz \Big \vert = \Big \vert \frac{1}{\omega_N} \int_{\partial B(0,1)}  h \eta_j d\sigma \Big \vert \leq  
\frac{\alpha_N}{\omega_N} \frac{1}{\alpha_N} \int_{\partial B(0,1)}  h d\sigma = N h(0),
\end{equation*}
where $\omega_N$ denotes the measure of the unit ball $B(0,1)$ and $\alpha_N$ is the measure of $ \partial B(0,1)$.  

From the latter and \eqref{comparaison_h(0)_et_dérivée_normale} we infer that 
\begin{equation*}
	|\nabla h(0)| \leq C(N) \Big|\frac{\partial h}{\partial \eta}(x)\Big|.
\end{equation*}
This proves the Lemma when $ N\geq3$.

For $N=2$, we observe that for any  $y \in \partial B(0,1/2)$, 
\begin{equation}\label{h_sur_le_bord_la_demi_boule-bis}
	h(y) \geq \disp\inf_{\partial B(0,1/2)} h \geq \frac{h(0)}{C_{H}}  \frac{v(y)}{\ln(2) }= \frac{h(0)}{C_{H}\ln(2)}v(y), 
\end{equation}
then the claim follows by the same procedure above, with $\varepsilon=\frac{h(0)}{\ln(2) C_{H}}>0$.
\end{proof}

\medskip

\begin{lem} \label{Brandt-weak}
	Let $N \geq 2$, $U \subset \R^{N}$ be a domain and $u \in C^{1}(U)$ be a solution of
	\begin{equation}\label{equation_de_poisson}
		\begin{array}{ccc}
			-\Delta u=f & \text{in} & \mathcal{D}'(U),
		\end{array}
	\end{equation}
	with $f \in L^{\infty}_{\text{loc}}(U)$. \\
	Assume $\delta >0$ and $y \in U$ such that $\overline{B(y,\delta)} \subset U$.  Then, for any 
	$x \in B(y,\frac{\delta}{2})$ we have
	\begin{equation}\label{inégalité_de_brandt}
		|u(x)-u(y)| \leq \dfrac{\sqrt{N}}{2^{1-\gamma}}\Big(2N\|u\|_{L^{\infty}(\overline{B(y,\delta)})} \delta^{-\gamma}+
		\|f\|_{L^{\infty}(\overline{B(y,\delta)})}\delta^{2-\gamma}\Big)|x-y|^{\gamma},
	\end{equation}
	with $\gamma \in ]0,1]$.
\end{lem}

\begin{proof}  
Let $y \in U$ and $R>0$ such that $\overline{B(y,R)} \subset U$ then, for any $ j = 1,...,N$, we have 
\begin{equation}\label{grad-est-Brandt}
	\Big \vert  \frac{\partial u}{\partial x_j}(y) \Big \vert  \leq \Big(\frac{N}{R}\|u\|_{L^{\infty}(\partial B(y,R))} + \frac{R}{2} \|f\|_{L^{\infty}(\overline{B(y,R)})} \Big).
\end{equation}	
This is exactly the same as the estimate $(3.15)$ on page 38 of \cite{gt}. Indeed, since $u$ is $C^1$, the barrier function 
$\Psi$ is smooth and $f$ belongs to $L^{\infty}_{\text{loc}}(U)$, we can apply the weak maximum principle (see, for example, Theorem 8.1 in \cite{gt}) to complete the proof in \cite{gt} and thus obtain the gradient estimate above \eqref{grad-est-Brandt}.

By the mean value theorem, for any $\gamma \in ]0,1]$,
	\begin{equation}\label{Mean_value_theorem}
		\begin{split}
			|u(x)-u(y)| \leq \disp\sup_{z \in \overline{B(y,\frac{\delta}{2})}} |\nabla u(z)| |x-y| \leq 
			\frac{\delta^{1-\gamma}}{2^{1-\gamma}}\disp\sup_{z \in \overline{B(y,\frac{\delta}{2})}}|\nabla u(z)| |x-y|^{\gamma}.
		\end{split}
	\end{equation}
	Moreover,  for any $z \in \overline{B(y,\frac{\delta}{2})}$ we have $\overline{B(z,\frac{\delta}{2})}\subset \overline{B(y,\delta)}\subset U$,  so we can apply \eqref {grad-est-Brandt} (with $R=\frac{\delta}{2}$) to get 
	\begin{equation}\label{majoration_du_gradient}
		|\nabla u(z)|\leq \sqrt{N}\Big(\frac{2 N}{\delta}\|u\|_{L^{\infty}(\overline{ B(y,\delta)})}+\frac{\delta}{4} \|f\|_{L^{\infty}(\overline{B(y,\delta)})}\Big),
	\end{equation}
	 for any $z \in \overline{B(y,\frac{\delta}{2})}$. The claim then follows by combining \eqref{Mean_value_theorem} and \eqref{majoration_du_gradient}.
 \end{proof}

\medskip

We conclude with the following useful extension lemma (in which we keep an explicit track of all constants and their dependencies). 

\medskip

\begin{lem}\label{extension_C_1-cylinder}
Assume $ N \geq 2$, $ \alpha \in (0,1] $ and $ r>0$. \\ Let $g \in C^1(\R^{N-1})$ be such that $\nabla g \in 
C^{0,\alpha}_{\text{loc}}(\R^{N-1}) $ and let us denote by $\Omega$ the epigraph defined by $g$. \\

(i) For any  $v \in C^{1,\alpha}(\overline{ \Omega \cap (B'(0',r)\times \R}))$ there exists $\widetilde{w} \in C^{1,\alpha}(\overline{B'(0',r)\times \R})$ such that 
	\begin{equation*}
		\widetilde{w}=v \quad \text{in } \quad \overline{ \Omega \cap (B'(0',r)\times \R})
	\end{equation*} 
	and
	\begin{equation*}
		\|\widetilde{w}\|_{C^{1,\alpha}(\overline{B'(0',r)\times \R})}\leq C (1+\|\nabla g \|_{C^{0,\alpha}(\overline{B'(0',r)})})^{4} 
		\|v\|_{C^{1,\alpha}(\overline{ \Omega \cap (B'(0',r)\times \R}))}
	\end{equation*}
	where $C=C(N)$ is a positive constant.

\bigskip

 (ii) Assume $\widetilde{T}>\displaystyle\sup_{B'(0',r)} \vert g \vert $ and $v \in C^{1,\alpha}
(\overline{\mathfrak{C}^{g}(0',r,\widetilde{T})})$. \\ Then, for any $\displaystyle\sup_{B'(0',r)} \vert g \vert  < t < \widetilde{T} $ there exists $\widetilde{w_t} \in C^{1,\alpha}(\overline{B'(0',r) \times (-t,t)}) $ such that 
	\begin{equation*}
		\widetilde{w_t}=v \quad \text{in } \quad \overline{\mathfrak{C}^{g}(0',r,t) }
	\end{equation*} 
	and
	\begin{equation*}
		\| \widetilde{w_t} \|_{C^{1,\alpha}(\overline{B'(0',r) \times (-t,t))} } \leq C (1+\|\nabla g \|_{C^{0,\alpha}(\overline{B'(0',r)})})^{4} \|v\|_{C^{1,\alpha}(\overline{\mathfrak{C}^{g}(0',r,\widetilde{T})})}
	\end{equation*}
	where $C=C(N,t, \widetilde{T})$ is a positive constant.
\end{lem}

\begin{proof}
	To simplify the presentation, we set $B'_{r}=B'(0',r)$ and $\Omega_r=\Omega\cap (B'(0',r)\times \R)$. Then we consider the following two functions :
\begin{equation*}
		\begin{array}{lccc}
			\psi : & \R^{N} & \longrightarrow & \R^{N} \\
			& x=(x',x_{N}) & \longmapsto & (x',x_{N}-g(x')) 
		\end{array}
	\end{equation*}
and	
\begin{equation*}
		\begin{array}{lccc}
			\theta : & \R^{N} & \longrightarrow & \R^{N} \\
			& y=(y',y_{N}) & \longmapsto & (y',y_{N}+g(y')) 
		\end{array}
	\end{equation*}	
and we observe that $\psi^{-1}= \theta$, $\psi(\overline{\Omega_r})=\overline{B'_r\times \R_{+}}$, 
$\psi(\partial \Omega\cap (B'_{r}\times \R))=\partial \R^{N}_{+} \cap (B'_r \times \R)$ and $\psi((B'_r \times \R) \backslash \overline{\Omega_r})=(B'_r \times \R)\backslash \overline{\R^{N}_{+}}$. Moreover, 
	\begin{equation}\label{prop-theta-psi}
\begin{split}		
&	|\theta(x)-\theta(y)| \leq |x-y|+|g(x')-g(y')|\leq (1+ \Vert \nabla g \Vert_{L^{\infty}(B'_r)})|x-y|, \quad \forall \, x,y \in \overline{B'_r\times \R},\\
&	|\psi(x)-\psi(y)| \leq |x-y|+|g(x')-g(y')|\leq (1+ \Vert \nabla g \Vert_{L^{\infty}(B'_r)} ) |x-y|,  \quad \forall \, x,y \in \overline{B'_r\times \R}.
\end{split}		
	\end{equation} 

\medskip

\noindent (i) For any $ y \in \overline{B'_r\times \R_{+}}$ set $\widetilde{v}(y)=v(\theta(y))$, then 
\begin{equation*}
		\widetilde{v} \in C^{1,\alpha}(\overline{B'_r\times \R_{+}}).
	\end{equation*}
	
	Indeed, by definition, we have $\widetilde{v} \in  C^0(\overline{B'_r\times \R_{+}})$ and 
\begin{equation}\label{bound_v_tilde}
		\|\widetilde{v}\|_{L^{\infty}(B'_r\times \R_{+})} = \|v\|_{L^{\infty}(\Omega_r)} 
\end{equation}
moreover, for any $x,y \in B'_r \times \R_{+}$, 
\begin{equation}\label{hold_v_tilde}
		\begin{split}
			|\widetilde{v}(x)-\widetilde{v}(y)|&\leq  \left[ v \right]_{C^{0,\alpha}(\Omega_r)} \vert  \theta (x) - \theta(y)\vert^{\alpha} \leq \left[ v \right]_{C^{0,\alpha}(\Omega_r)}  (1+ \Vert \nabla g \Vert_{L^{\infty}(B'_r)})^{\alpha}|x-y|^{\alpha},
		\end{split} 
	\end{equation}
where in the latter we have used \eqref{prop-theta-psi}. Hence, 
\begin{equation}\label{hold_v_tilde-bis}
		\begin{split}
		 \left[ \widetilde{v}\right]_{C^{0,\alpha}(B'_r \times \R_{+})} \leq \left[ v \right]_{C^{0,\alpha}(\Omega_r)}  (1+ \Vert \nabla g \Vert_{L^{\infty}(B'_r)})^{\alpha}.
		\end{split} 
	\end{equation}
	
Also, for any $i \in \{1,\cdots, N-1\}$,
	\begin{equation*}
		\partial_{i}\widetilde{v}(y)=\partial_{i}v(\theta (y))+\partial_{i}g(y')\partial_{N}v(\theta (y))
	\end{equation*}
	and
	\begin{equation*}
		\partial_{N}\widetilde{v}(y)=\partial_{N}v (\theta (y)),
	\end{equation*}
	thus,
\begin{equation}\label{bound_nabla_v_tilde}
		\| \nabla \widetilde{v} \|_{L^{\infty}(B'_r \times \R_{+} )} \leq \sqrt{2}
		(1+ \| \nabla g \|_{L^{\infty}(B'_r)}) \| \nabla v\|_{L^{\infty}(\Omega_r)}.
\end{equation}
Furthermore, for any $x,y \in B'_r \times \R_{+}$ and for any $i \in \lbrace 1,\cdots, N-1 \rbrace $, we have 
	\begin{equation}\label{derivative_for_i_different_from_N}
		\begin{split}
			|\partial_{i}\widetilde{v}(x)-\partial_{i} \widetilde{v}(y)|&\leq |\partial_{i}v(\theta(x))-\partial_{i}v(\theta(y))|+ 
			|\partial_{i}g(x')\partial_{N}v(\theta(x))-\partial_{i}g(y')\partial_{N}v(\theta(y))|\\
			& \leq \left[ \partial_{i}v \right]_{C^{0,\alpha}(\overline{\Omega_r})} \vert  \theta (x) - \theta(y)\vert^{\alpha}  + 
			|\partial_{i}g(x') - \partial_{i} g(y') \vert \partial_{N}v(\theta(x)) \vert \\
			& + |\partial_{i}g(y')\vert \vert \partial_{N}v(\theta(x))- \partial_{N}v(\theta(y))| \leq  \left[ \partial_{i}v 
			\right]_{C^{0,\alpha}(\overline{\Omega_r})} \vert  \theta (x) - \theta(y)\vert^{\alpha} \\
			& + \vert \partial_{N}v(\theta(x)) \vert\left[ \partial_{i}g \right]_{C^{0,\alpha}(\overline{B'_r})}\vert x'-y' \vert^{\alpha} + 
			|\partial_{i}g(y')\vert 
			\left[ \partial_{N}v \right]_{C^{0,\alpha}(\overline{\Omega_r}))} \vert  \theta (x) - \theta(y)\vert^{\alpha} \\
			& \leq \left[ \partial_{i}v \right]_{C^{0,\alpha}(\overline{\Omega_r})} (1+ \Vert \nabla g \Vert_{L^{\infty}(B'_r)})^{\alpha}
			\vert x-y\vert^{\alpha} + \Vert \partial_{N}v \Vert_{L^{\infty}(\Omega_r)} 
			\left[ \partial_{i}g \right]_{C^{0,\alpha}(\overline{B'_r})}\vert x'-y' \vert^{\alpha}  \\
			&   + \Vert \partial_{i}g \Vert_{L^{\infty}(B'_r)} \left[ \partial_{N}v \right]_{C^{0,\alpha}(\overline{\Omega_r})} 
			    (1+ \Vert \nabla g \Vert_{L^{\infty}(B'_r)})^{\alpha}\vert x-y\vert^{\alpha} \\
			& \leq \left(  ( \left[ \partial_{i}v \right]_{C^{0,\alpha}(\overline{\Omega_r})} + \Vert \partial_{i}g \Vert_{L^{\infty}(B'_r)} 
			\left[ \partial_{N}v \right]_{C^{0,\alpha}(\overline{\Omega_r})})
			 (1+ \Vert \nabla g \Vert_{L^{\infty}(B'_r)})^{\alpha} \right) \vert x-y\vert^{\alpha} + \\
			& + \Vert \partial_{N}v \Vert_{L^{\infty}(\Omega_r)} \left[ \partial_{i}g \right]_{C^{0,\alpha}(\overline{B'_r})} 
			\vert x-y \vert^{\alpha}  \\
			& \leq (1+ \Vert \nabla g \Vert_{L^{\infty}(B'_r)})^{\alpha+1} 
			( \left[ \partial_{i}v \right]_{C^{0,\alpha}(\overline{\Omega_r})}+
			\left[ \partial_{N}v \right]_{C^{0,\alpha}(\overline{\Omega_r})}) \vert x-y\vert^{\alpha} + \\
			& + \Vert \partial_{N}v \Vert_{L^{\infty}(\Omega_r)} \left[ \partial_{i}g \right]_{C^{0,\alpha}(\overline{B'_r})} 
			\vert x-y \vert^{\alpha}  \\
			& \leq 2  (1+ \Vert \nabla g \Vert_{C^{0,\alpha}(\overline{B'_r})})^{\alpha+1} 
			\Vert  \nabla v \Vert_{C^{0,\alpha}(\overline{\Omega_r})} \vert x-y \vert^{\alpha}. \\
		\end{split} 
	\end{equation}
Hence,
\begin{equation}\label{derivative_for_i_different_from_N-bis}
		\begin{split}
			\left[ \partial_{i} \widetilde{v} \right]_{C^{0,\alpha}(\overline{B'_r \times \R_{+}})} \leq 2  (1+ \Vert \nabla g 
			\Vert_{C^{0,\alpha}(\overline{B'_r})})^{\alpha+1} \Vert  \nabla v \Vert_{C^{0,\alpha}(\overline{\Omega_r})}.
		\end{split} 
	\end{equation} 
Similarly we have 
	\begin{equation}\label{derivative_for_N}
		|\partial_{N}\widetilde{v}(x)-\partial_{N} \widetilde{v}(y)| \leq \left[ \partial_{N} v \right]_{C^{0,\alpha}(\overline{\Omega_r})} 
		(1+ \Vert \nabla g \Vert_{L^{\infty}(B'_r)})^{\alpha}\vert x-y \vert^{\alpha}
	\end{equation}
and so
\begin{equation}\label{derivative_for_N-bis}
		\left[ \partial_{N} \widetilde{v} \right]_{C^{0,\alpha}(\overline{B'_r \times \R_{+}})} \leq \left[ \partial_{N} v \right]_{C^{0,\alpha}(\overline{\Omega_r})} (1+ \Vert \nabla g \Vert_{L^{\infty}(B'_r)})^{\alpha}.
	\end{equation}

Thanks to \eqref{bound_v_tilde}, \eqref{hold_v_tilde-bis}, \eqref{bound_nabla_v_tilde},\eqref{derivative_for_i_different_from_N-bis} and \eqref{derivative_for_N-bis} we deduce that $\widetilde{v} \in C^{1,\alpha}(\overline{B'_r \times \R_{+}})$ and 
	\begin{equation}\label{estimate_v_tilde}
		\|\widetilde{v}\|_{C^{1,\alpha}(\overline{B'_r \times \R_{+})}} \leq (2N+1) ( 1 +
		\| \nabla g \|_{C^{0,\alpha} (\overline{B'_r})})^{\alpha+1} \|v\|_{C^{1,\alpha}(\overline{\Omega_r})}.
	\end{equation}

Now, for any $ y \in \overline{B'_r \times \R}$ we set 
	\begin{equation*}
		\widetilde{V}(y)=
		\left\{
		\begin{array}{ccc}
			\widetilde{v}(y',y_{N})  & \text{if} & y_{N}\geq 0,\\
			\disp\sum_{i=1}^{2}c_{i}\widetilde{v} \left( y',-\frac{y_{N}}{i} \right)  & \text{if}& y_{N}<0,
		\end{array}
		\right.
	\end{equation*}
	where $c_{1}=-3$ and $c_{2}=4$. With this choice, we see that $\widetilde{V}$ belongs to $C^1(B'_r \times \R)$ and 
	\begin{equation}\label{norm_C1_V_tilde}
		\|\widetilde{V}\|_{L^{\infty}(B'_r \times \R)}+ \|\nabla \widetilde{V}\|_{L^{\infty}(B'_r \times \R)}  \leq 
		7(N+1) \| \widetilde{v}\|_{C^{1,\alpha}(\overline{B'_r \times \R_{+}})}.
	\end{equation}
	Now, using the value of $c_1$ and $c_2$, and   
	\begin{equation*}
		|x_{N}+\frac{y_{N}}{i}|\leq |x_{N}|+\frac{|y_{N}|}{i}\leq |x_{N}|+|y_{N}|\leq x_{N}-y_{N}=|x_{N}-y_{N}|,
	\end{equation*}
it follows that, for any $i \in \{1, \cdots ,N\} $
\begin{equation}\label{holder_deriv-V_tilde}
\left[ \partial_{i} \widetilde{V} \right]_{C^{0,\alpha}(\overline{B'_r \times \R})} \leq \ C_1(N) 
\| \widetilde{v}\|_{C^{1,\alpha}(\overline{B'_r \times \R_{+}})}.
\end{equation}	
	
	Hence, \eqref{estimate_v_tilde}, \eqref{norm_C1_V_tilde} and \eqref{holder_deriv-V_tilde}  lead to 
	\begin{equation}\label{norme_v_tilde}
		\|\widetilde{V}\|_{C^{1,\alpha}(\overline{B'_r \times \R})} \leq 
		C_2(N) (1+\| \nabla g \|_{C^{0,\alpha} (\overline{B'_r})})^{\alpha+1}  \| v \|_{C^{1,\alpha}(\overline{\Omega_r})}.
	\end{equation}
	
To proceed further, for any $ y \in \overline{B'_r\times \R}$ set $\widetilde{w}(y):=\widetilde{V}(\psi(y))$, then $ \widetilde{w}$ satisfies  the claim. 
Indeed, if $y \in \overline{\Omega_r}$ then $ \psi(y) \in \overline{B'_r \times \R_{+}},$ and so $\widetilde{V}(\psi(y)) = \widetilde{v}(\psi(y))$. Hence, $ \widetilde{w}(y) = \widetilde{V}(\psi(y))= \widetilde{v}(\psi(y)) = v((\theta \circ \psi) (y)) = v(y)$ for any $y \in \overline{\Omega_r}$.  Also, in view of the definition of $\psi$ and \eqref{prop-theta-psi}, the same computations leading to \eqref{estimate_v_tilde} yield 
	
\begin{equation}\label{estimate_W_tilde}
		\|\widetilde{w}\|_{C^{1,\alpha}(\overline{B'_r \times \R})} \leq (2N+1) 
		(1+\| \nabla g \|_{C^{0,\alpha} (\overline{B'_r})})^{\alpha+1} \|\widetilde{V}\|_{C^{1,\alpha}( \overline{B'_r \times \R})} .\end{equation}	
		
By putting together  \eqref{estimate_W_tilde} and \eqref{norme_v_tilde} we get 

\begin{equation}\label{estimate_W_tilde-finale}
\|\widetilde{w}\|_{C^{1,\alpha}(\overline{B'_r \times \R})} \leq  
		C_3(N)(1+\| \nabla g \|_{C^{0,\alpha} (\overline{B'_r})})^{2(\alpha+1)}  \| v \|_{C^{1,\alpha}(\overline{\Omega_r})}.
\end{equation}	
	
This concludes the proof of the first part of the Lemma. 	
	
\medskip

\noindent (ii) Let us consider an even function $\eta_t \in C_c^2(\R)$ such that $ 0 \le \eta_t \le 1,$ 
\begin{equation*}\label{cut-off-1D}
\eta_t(x_N) := \begin{cases}
1 \quad & {\rm if} \qquad 0 \leq x_N \leq t ,\\
0 & {\rm if} \qquad x_N \ge t + \frac{\widetilde{T} -t }{2}.
\end{cases}
\end{equation*}
Then,  $v \eta_t \in  C^{1,\alpha}(\overline{\Omega_r})$ and $ \|  v \eta_t \|_{C^{1,\alpha}(\overline{\Omega_r})} =
\|  v \eta_t \|_{C^{1,\alpha}(\overline{\mathfrak{C}^{g}(0',r,\widetilde{T})})} 
\leq C(\eta_t) \| v \|_{C^{1,\alpha}(\overline{\mathfrak{C}^{g}(0',r,\widetilde{T})})}$.

\noindent From the first part of Lemma we get $\widetilde{v \eta_t}\in C^{1,\alpha}(\overline{B'_r \times \R})$ such that 
$\widetilde{v \eta_t} = v \eta_t$ on  $\overline{\Omega_r}$ and 
\begin{equation*}
           \begin{split}
\|\widetilde{v \eta_t}\|_{C^{1,\alpha}(\overline{B'_r \times \R})} & \leq C_3(N)(1+\| \nabla g \|_{C^{0,\alpha} (\overline{B'_r})})^{4}  \| v \eta_t \|_{C^{1,\alpha}(\overline{\Omega_r})} \\
& \leq C(\eta_t) C_3(N)(1+\| \nabla g \|_{C^{0,\alpha} (\overline{B'_r})})^{4}  
\| v \|_{C^{1,\alpha}(\overline{\mathfrak{C}^{g}(0',r,\widetilde{T})})}
           \end{split}
\end{equation*}

\noindent To conclude, we set $ \widetilde{w}= \widetilde{v \eta_t}_{\vert \overline{B'_r \times (-t,t)}}$ and we observe that, the positive constant $C(\eta_t)$ depends only on $t$ and $ \widetilde{T}$. \end{proof}

\bigskip

\subsection{Notations}\label{Notations} \quad \\

\noindent  $\mathbb{R}^N_+ =\{x=(x', x_N) \in \R^{N-1} \times \R  \ | \ x_N>0\}$, the open upper half-space of $\R^N$.  

\bigskip

\noindent $ \vert \cdot \vert$ : the Euclidean norm. 

\bigskip

\noindent $B(x,R)$  : the Euclidean $N$-dimensional open ball of center $x$ and radius $R>0$. 

\bigskip

\noindent $B'(x',R)$  : the Euclidean $N-1$-dimensional open ball of center $x'$ and radius $R>0$. 

\bigskip

\noindent $B_R : = B(0 , R)$  and $B'_R := B'(0', R)$, where $0 =(0',0) \in \R^{N-1} \times \R$ is the origin of $ \R^N$.  

\bigskip

\noindent $\mathcal{H}^{N-1}$ : the $N-1$-dimensional Hausdorff measure. 

\bigskip

\noindent $UC(X)$ : the set of uniformly continuous functions on $X$. 

\bigskip

\noindent $Lip(X)$ : the set of globally Lipschitz-continuous functions on $X$. 

\bigskip

\noindent $Lip_{loc}(X)$ : the set of locally Lipschitz-continuous functions on $X$. 

\bigskip

\noindent $C^k(\overline{U})$ : the set of functions in $C^k(U)$ all of whose derivatives of order $ \leq k$ have continuous (not necessarily bounded) extensions to the closure of the open set $U$. 

\bigskip

\noindent $C^{0,\alpha}(\overline{U})$ : the vector space of bounded and globally $\alpha$-Hölder-continuous functions $h$ on the open set $U$ endowed with the norm :

\medskip

\qquad $\Vert h \Vert_{C^{0,\alpha}(\overline{U})} := \|h\|_{L^{\infty}(U)} + 
\left[ h \right]_{C^{0,\alpha}  (U)} :=  \sup_{x \in U} |h(x)| + \sup_{x,y \in U, x\neq y}
\frac{|h(x)-h(y)|}{|x-y|^{\alpha}} $ . 

\bigskip

\noindent $C^{k,\alpha}(\overline{U})$ : the vector space of functions in $C^k(U)$ all of whose derivatives of order $ \leq k$ belong to $C^{0,\alpha}(\overline{U})$, endowed with the norm :

\medskip

\hskip4truecm  $\Vert h \Vert_{C^{k,\alpha}(\overline{U})} := \sum_{0 \leq \vert \beta \vert \leq k} \Vert \partial^{\beta} h 
\Vert_{C^{0,\alpha}(\overline{U})}$ . 

\bigskip

\noindent $\mathcal{D}'(U)$ : the space of distributions on the open set $U$. 

\bigskip

\noindent $H^1_{loc}(\overline{U}) = \{ u : U \mapsto \R,  \,\,  \text{$u$ Lebesgue-mesurable} \, : \, u \in H^1(U \cap B(0,R)) \quad \forall \, R>0 \}$, \\
 i.e., $u$ is Lebesgue-measurable on the open set $U$ and $u \in H^1(V)$ for any open bounded set $ V\subset U$.
 
 \bigskip

\medskip

\section*{Acknowledgments} 
\noindent The authors would like to thank Berardino Sciunzi for many stimulating and valuable discussions on the topics covered in this paper. 

\bibliographystyle{sorting=nyt}

\begin{thebibliography}{99}

\bibitem{BCN3} {\sc H. Berestycki, L.A. Caffarelli, L. Nirenberg.}
\emph{Monotonicity for elliptic equations in an unbounded Lipschitz domain.} Comm. Pure Appl. Math. 50, 1997, 1089-1111.

\bibitem{bfs}{\sc N. Beuvin, A. Farina, B. Sciunzi.} 
\emph{Monotonicity for solutions to semilinear problems in epigraphs.} J. Math. Pures Appl. (9) 210 (2026), Paper No. 103868, 44 pp.


\bibitem{dan1}{\sc E.N. Dancer.} 
\emph{Stable and finite morse index solutions on $\R^{N}$ or on bounded domain with small diffusion.} Transactions of the American Mathematical Society , 357,  3, 1225-1243.


\bibitem{ppw}{\sc M. Del Pino, F. Pacard, J. Wei.} 
\emph{Serrin's overdetermined problem and constant mean curvature surfaces.} Duke Math. J., 164, 14, 2015.
	
\bibitem{Esteban-Lions} {\sc M.J. Esteban, P.-L. Lions.}
\emph{Existence and nonexistence results for semilinear elliptic problems in unbounded domains.}  Proc. Roy. Soc. Edinburgh (1982), 1–14.
	
	
\bibitem{farBer} {\sc A. Farina.} 
\emph{A sharp Bernstein-type theorem for entire minimal graphs.}   Calc. Var. Partial Differential Equations 57 (2018), no. 5, 5 pp.	


\bibitem{FMV} {\sc A. Farina, L. Mari, E. Valdinoci.}
\emph{Splitting Theorems, Symmetry Results and Overdetermined Problems for Riemannian Manifolds.}
Comm. Partial Differential Equations 38 (2013), no. 10, 1818-1862. 

	
\bibitem{fs1} {\sc A. Farina, B. Sciunzi.}
\emph {Qualitative properties and classification of nonnegative solutions to $ -\Delta u=f(u) $ in 
unbounded  domains when $f(0)<0$.} Rev. Matematica Iberoamericana, 32 (2016), no. 4, 1311-1330. 

\bibitem{fs2} {\sc A. Farina, B. Sciunzi.}
\emph {Monotonicity and symmetry of nonnegative solutions to $-\Delta u = f(u)$ in half-planes and strips.} 
Adv. Nonlinear Stud. 17 (2017), no. 2, 297-310.
	
	
\bibitem{fsv} {\sc A. Farina, B. Sciunzi, E. Valdinoci.}
\emph{Bernstein and De Giorgi type problems: new results via a geometric approach.}
Ann. Sc. Norm. Super. Pisa Cl. Sci. (5),7, 4, 2008, 741-791.
	
\bibitem{farval} { \sc A.Farina, E. Valdinoci.}
\emph{A pointwise gradient estimate in possibly unbounded domains with nonnegative mean curvature.} Advances. In. Mathematics, 225, 2010, 2808-2827.
	
\bibitem{FVarma} {\sc A. Farina, E. Valdinoci.}
\emph{Flattening results for elliptic {PDE}s in unbounded domains with applications to overdetermined problems.}
Arch. Ration. Mech. Anal., 195 (3), 2010, 1025-1058.
	
	
\bibitem{gt}{\sc D. Gilbarg, N. Trudinger.}
\emph{Elliptic partial differential equations of Second order.} 2nd edition, Springer-Verlag, 1983.
	
\bibitem{Giu}{\sc E. Giusti.} 
\emph{Minimal Surfaces and Functions of Bounded Variation.}  Monographs in Mathematics, 
Vol. 80, Birkhäuser Verlag, Basel, (1984).

\bibitem{GMM}{\sc E. Gonzalez, U. Massari, M. Miranda.} 
\emph{On minimal cones.}  Appl. Anal. 65 (1997), no. 1-2, 135–143.

\bibitem{RRS17} {\sc A. Ros,  D. Ruiz, P. Sicbaldi.}  
\emph{A rigidity for overdetermined elliptic problems in the plane.} Comm. Pure Appl. Math. 70 (2017) 1223–1252.

	
\bibitem{ser}{\sc J. Serrin.} 
\emph{A symmetry problem in potential theory.} Arch.  Ration. Mech. Anal., 43, 4, 1971, 304-318.

\bibitem{Ser-Puc}{\sc J. Serrin, P. Pucci} \emph{The maximum principle.} Progress in Nonlinear Differential Equations and their Applications, 73, Birkh\"auser, Basel, 2007.
	
\bibitem{Vogel} {\sc A.L. Vogel}
\emph{Symmetry and regularity for general regions having a solution to certain overdetermined boundary value problems.}
Atti Sem. Mat. Fis. Univ. Modena, XL, 443-484 (1992). 
	
	
\bibitem{ww}{\sc K. Wang, J. Wei.} 
\emph{ On Serrin's overdetermined problem and a conjecture of Berestycki, Caffarelli and Nirenberg.} Comm. Partial Differential Equations 44 (2019) 837–858.
	

\end{thebibliography}

\end{document}